\documentclass[11pt]{amsart}
\usepackage{amsmath}
\usepackage{amssymb}
\usepackage{amsthm}

\newcommand{\aut}{\operatorname{Aut}}

\newcommand{\soc}{\operatorname{soc}}
\newcommand{\FF}[1]{\mathbb F_{#1}}

\newcommand{\sym}{\operatorname{Sym}}
\newcommand{\alt}{\operatorname{Alt}}

\newcommand{\nor}{\vartriangleleft}
\newcommand{\la}{\lambda}
\newcommand{\id}{\textup{id}}
\newcommand{\ol}{\bar}

\newcommand{\ind}{\operatorname{Ind}}
\newcommand{\ra}{\rightarrow}
\newcommand{\fX}{\mathfrak X}

\newtheorem{thm}{Theorem}[section]
\newtheorem{prop}[thm]{Proposition}
\newtheorem{lem}[thm]{Lemma}
\newtheorem{cor}[thm]{Corollary}

\theoremstyle{definition}
\newtheorem{defin}[thm]{Definition}

\theoremstyle{definition}
\newtheorem{remark}[thm]{Remark}

\begin{document}

\title{A proof of Pyber's base size conjecture}

\author{H\"ulya Duyan}
\address{Department of Mathematics, Central European University, N\'ador utca 9., H-1051, Budapest, Hungary}
\email{duyan\_hulya@phd.ceu.edu}

\author{Zolt\'an Halasi} 
\address{Department of Algebra and Number Theory, E\"otv\"os University, P\'azm\'any P\'eter s\'et\'any 1/c, H-1117, Budapest, Hungary and Alfr\'ed R\'enyi Institute of Mathematics, Hungarian Academy of Sciences, Re\'altanoda utca 13-15, H-1053, Budapest, Hungary}
\email{zhalasi@cs.elte.hu and halasi.zoltan@renyi.mta.hu}

\author{Attila Mar\'oti}
\address{Alfr\'ed R\'enyi Institute of Mathematics, Hungarian Academy of Sciences, Re\'altanoda utca 13-15, H-1053, Budapest, Hungary}
\email{maroti.attila@renyi.mta.hu}

\date{\today}


\thanks{The second author was supported from the ERC Limits of
  discrete structures Grant No.~617747. The third author has received
  funding from the European Research Council (ERC) under the European
  Uni\-on's Horizon 2020 research and innovation program (grant
  agreement No. 648017) and was supported by the J\'anos Bolyai
  Research Scholarship of the Hungarian Academy of Sciences. The
  second and third authors were also supported by a Humboldt Return
  Fellowship and by National Research, Development and Innovation
  Office (NKFIH) Grant No.~K115799.}

\begin{abstract}
  Building on earlier papers of several authors, we establish that
  there exists a universal constant $c > 0$ such that the minimal base
  size $b(G)$ of a primitive permutation group $G$ of degree $n$
  satisfies $\log |G| / \log n \leq b(G) < 45 (\log |G| / \log n) +
  c$. This finishes the proof of Pyber's base size conjecture. The
  main part of our paper is to prove this statement for affine
  permutation groups $G=V\rtimes H$ where $H\leq GL(V)$ is an
  imprimitive linear group. An ingredient of the proof is that for the
  distinguishing number $d(G)$ (in the sense of Albertson and Collins)
  of a transitive permutation group $G$ of degree $n > 1$ we have the
  estimates $\sqrt[n]{|G|} < d(G) \leq 48 \sqrt[n]{|G|}$.
\end{abstract}

\maketitle 

\section{Introduction}

Let $G$ be a permutation group acting on a finite set $\Omega$ of size
$n$. A subset $\Sigma$ of $\Omega$ is called a base for $G$ if the
intersection of the stabilizers in $G$ of the elements of $\Sigma$ is
trivial. Bases played a key role in the development of permutation
group theoretic algorithms. For an account of such algorithms see the
book of Seress \cite{Seressbook}. Since these algorithms are generally
faster and require less memory if the size of the base is small, it is
fundamentally important to find a base of small size.

The minimal size of a base for $G$ on $\Omega$ is denoted by
$b(G)$. Blaha \cite{Blaha} showed that the problem of finding $b(G)$
for a permutation group $G$ is NP-hard. One may approximate $b(G)$ by
a greedy heuristic; always choose a point from $\Omega$ whose orbit is
of largest possible size under the action of the intersection of the
stabilizers in $G$ of the previous points chosen. Blaha \cite{Blaha}
proved that the size of such a base is $O(b(G) \log \log n)$ and that
this bound is sharp. (Here and throughout the paper the base of the
logarithms is $2$ unless otherwise stated.) On the other hand, Pyber
\cite{pyber} showed that there exists a universal constant $c > 0$
such that almost all (a proportion tending to $1$ as $n \to \infty$)
subgroups $G$ of $\sym(n)$ satisfy $b(G) > c n$.

The minimal base size of a primitive permutation group $G$ of degree $n$
not containing $\alt(n)$ has been widely studied.  Already in the
nineteenth century Bochert \cite{Bochert} showed that $b(G) \leq n/2$
for such a group $G$. This bound was substantially improved by Babai to
$b(G) < 4 \sqrt{n} \log n$, for uniprimitive groups $G$, in
\cite{BabaiAnnals}, and to the estimate $b(G) < 2^{c \sqrt{\log n}}$
for a universal constant $c > 0$, for doubly transitive groups $G$, in
\cite{BabaiInvent}. The latter bound was improved by Pyber
\cite{PyberDoubly} to $b(G) < c {(\log n)}^{2}$ where $c$ is a
universal constant. These estimates are elementary in the sense that
their proofs do not require the Classification of Finite Simple Groups
(CFSG). Using CFSG, Liebeck \cite{Liebeck84} classified all primitive
permutation groups $G$ of degree $n$ with $b(G) \geq 9 \log n$.

Let $G$ be an almost simple primitive permutation group. We say that
$G$ is standard if either $G$ has alternating socle $\alt(m)$ and the
action is on subsets or partitions of $\{ 1, \ldots, m \}$, or $G$ is
a classical group acting on an orbit of subspaces (or pairs of
subspaces of complementary dimension) of the natural module. Otherwise
$G$ is said to be non-standard. A well-known conjecture of Cameron and
Kantor \cite{CameronKantor} asserts that there exists an absolute constant
$c$ such that $b(G) \leq c$ for all non-standard primitive permutation
groups $G$. In case $G$ has an alternating socle, this was established
by Cameron and Kantor \cite{CameronKantor}. Later in
\cite[p. 122]{Cameron} Cameron writes that $c$ can probably be taken
to be $7$, and the only extreme case is the Mathieu group
$M_{24}$ in its natural action. The Cameron-Kantor conjecture was proved by Liebeck and
Shalev in \cite{LS99}, and Cameron's bound of $7$ was established in the
series of papers \cite{LiebeckShalev77}, \cite{LiebeckShalev78},
\cite{Burness3}, \cite{Burness5}, \cite{Burness6},
\cite{Burness4}. The proofs are probabilistic and use bounds on fixed
point ratios.


Let $d$ be a fixed positive integer. Let $\Gamma_{d}$ be the class of
finite groups $G$ such that $G$ does not have a composition factor
isomorphic to an alternating group of degree greater than $d$ and no
classical composition factor of rank greater than $d$. Babai, Cameron,
P\'alfy \cite{BCP} showed that if $G \in \Gamma_{d}$ is a primitive
permutation group of degree $n$, then $|G| < n^{f(d)}$ for some
function $f(d)$ of $d$. Babai conjectured that there is a function
$g(d)$ such that $b(G) < g(d)$ whenever $G$ is a primitive permutation
group in $\Gamma_{d}$. Seress \cite{Ser} showed this for $G$ a
solvable primitive group by establishing the bound $b(G)\leq 4$. Babai's
conjecture was proved by Gluck, Seress, Shalev \cite{GSSh} with a
bound $g(d)$ which is quadratic in $d$. Later, Liebeck and Shalev
\cite{LS99} showed that in Babai's conjecture the function $g(d)$ can
be taken to be linear in $d$.
 
Since any element of a permutation group $G$ is determined by its
action on a base, we clearly have $|G| \leq n^{b(G)}$ where $n =
|\Omega|$ is the degree of $G$. From this we get the estimate $\log
|G| / \log n \leq b(G)$.  An important question of Pyber \cite[Page
  207]{pyber} from 1993 asks if this latter bound is essentially sharp
for primitive permutation groups $G$. Specifically, he asked whether
there exists a universal constant $c > 0$ such that
\[b(G) < c \frac{\log |G|}{\log n}.\]

Pyber's conjecture is an essential generalization of the known upper
bounds for $b(G)$, the Cameron-Kantor conjecture,
and Babai's conjecture.

By the Aschbacher-O'Nan-Scott theorem, primitive permutation groups
fall in several types: almost simple type, diagonal type, product
type, twisted wreath product type, and affine type. Pyber's conjecture
has been verified for all non-affine primitive permutation groups. For
non-standard (almost simple) permutation groups Pyber's conjecture
follows from the proof of the Cameron-Kantor conjecture, and for
standard (almost simple) permutation groups Pyber's conjecture was
settled by Benbenishty in \cite{Benbenishty}. Primitive permutation
groups of diagonal type were handled by Gluck, Seress, Shalev
\cite[Remark 4.3]{GSSh} and Fawcett \cite{Fawcett}. For primitive
groups of product type and of twisted wreath product type the
conjecture was established by Burness and Seress \cite{BS}. From these
results one can deduce the general bound $$b(G) < 45 \frac{\log |G|}{\log
n}$$ for a non-affine primitive permutation group $G$ of degree
$n$. 

An affine primitive permutation group $G$ acting on a set $\Omega$ is
defined to be a primitive permutation group with a (unique) regular
abelian normal subgroup $V$. The subgroup $V$ is elementary
abelian. Identifying $\Omega$ with $V$, denote the stabilizer in $G$
of the zero vector by $H$. The group $H$ can be viewed as a subgroup
of $GL(V)$ and $G=V \rtimes H$ as a subgroup of $AGL(V)$.  Since $G$ is a
primitive permutation group, $H$ is maximal in $G$ and acts
irreducibly and faithfully on $V$. The action of $H$ on $V$ may or may
not preserve a non-trivial direct sum decomposition of the vector
space $V$. In the first case $V$ is said to be an imprimitive
$H$-module, and in the latter case $V$ is called a primitive
$H$-module. In this paper we will simply call $H$ an imprimitive
linear group or a primitive linear group if $V$ is imprimitive or
primitive, respectively. 

The most general result on the base size of affine primitive
permutation groups is due to Liebeck and Shalev \cite{LS1}, \cite{LS2}
who established Pyber's conjecture in the case where $H$
is a primitive linear group (see Theorem \ref{LieSha}). In this paper
we use a characterization of primitive linear groups of unbounded base
size given by Liebeck and Shalev \cite{LS1}, \cite{LS2} (see Theorem
\ref{thm:tensorK_1}). There is a similar characterization of primitive
linear groups of large orders due to Jaikin-Zapirain and Pyber
\cite[Proposition 5.7]{JaPy}.

In case $(|H|,|V|)=1$ for an affine primitive permutation group
$G=V \rtimes H$, Pyber's conjecture was first established by Gluck and Magaard
in \cite{GluckMagaard} by showing that $b(G)\leq 95$. In fact, in this
case the best possible result is $b(G)\leq 3$ proved by Halasi and
Podoski in \cite{HP}. Solvable or more generally, $p$-solvable
affine primitive permutation groups also satisfy Pyber's conjecture
(where $p$ is the prime divisor of the degree). In these cases, Seress
\cite{Ser} and Halasi and Mar\'oti \cite{HM} established the best
possible bound $b(G)\leq 4$. Fawcett and Praeger \cite{FP} proved
Pyber's conjecture for affine primitive permutation groups $G=V\rtimes H$ in
case $H$ preserves a direct sum decomposition $V=V_1\oplus\ldots\oplus
V_t$ where $H$ is close to a full wreath product $GL(V_1)\wr L$ with
$L$ a permutation group of degree $t$ satisfying any of four given
properties.

In this paper we complete the proof of Pyber's conjecture by handling
the case of affine primitive permutation groups $G=V \rtimes H$ where $V$ is
an imprimitive $H$-module. A stronger form of Pyber's conjecture is
the following.
\begin{thm}\label{Pyber}
  There exists a universal constant $c >0$ such that the minimal base
  size $b(G)$ of a primitive permutation group $G$ of degree $n$
  satisfies
  \[\frac{\log |G|}{\log n} \leq b(G) < 45
  \frac{\log |G|}{\log n} + c.\]
\end{thm}

The minimal base size of a permutation group is related to several
other invariants of the group. For example, Robinson \cite{Robinson}
showed that if $G$ is a primitive permutation group of degree $n$ and
rank $r$, then $b(G) \leq (n-1)/(r-1)$. The minimal degree $m$ of a
transitive permutation group of degree $n$ is also related to the minimal base
size $b$ by the inequality $mb\geq n$. 

There are at least two concepts termed by the name
``distinguishing number". Both of these are connected to the minimal
base size of a group. In 1981 Babai \cite{BabaiAnnals} defined the
distinguishing number of a coherent configuration and used it to
establish the aforementioned bound for the minimal base size. This
notion was later also used in a recent paper by Sun and Wilmes
\cite{SunWilmes}. In the present paper we use a different concept
with the same name. This different definition was introduced
for graphs in 1996 by Albertson and Collins \cite{AlbertsonCollins}
and since then many authors have used it under the name ``distinguishing
number". For more information, see Sections 2.2 and 3.4 of the
excellent survey article by Bailey and Cameron \cite{BC}.

For a permutation group $G$ acting on a finite set $\Omega$ we write
$d(G)$ for the minimal number of colors needed to color the elements
of $\Omega$ in such a way that the stabilizer in $G$ of this coloring
is trivial. This invariant is called the distinguishing number of the
permutation group. Seress \cite{Ser} proved that $d(G) \leq 5$ for a
solvable permutation group $G$. By results of Seress \cite{Seress} and
Dolfi \cite{Dolfi}, it follows that $d(G) \leq 4$ for a primitive
permutation group $G$ of degree $n$ which does not contain $\alt
(n)$. Clearly, if $G$ is a permutation group of degree $n > 1$, then
$\sqrt[n]{|G|} < d(G)$. Burness and Seress \cite{BS} stated (with a
different languague) that there exists a universal constant $c > 0$ such that
$d(G) \leq {|G|}^{c/n}$ provided that $G$ is a transitive permutation
group of degree $n$ (see also Theorem \ref{equiv} and the discussion
preceding it). The proof of this latter fact misses a case. In
this paper we show the following stronger result.
\begin{thm}\label{d(P)}
  Let $G$ be a transitive permutation group of degree
 $n > 1$. Then $\sqrt[n]{|G|} < d(G) \leq 48 \sqrt[n]{|G|}$.
\end{thm}
This result (and its proof) plays a key role in handling the ``top
action'' of an imprimitive irreducible linear group. For a rough idea
of this application, see Lemma \ref{lem:trivK_1}.

This paper is organized as follows.  
In Section 2 we examine the
distinguishing number of transitive permutation groups and we prove
Theorem 1.2. One of the intermediate results, namely, Theorem 2.8,
will also be used later in Section 3.3.

In Section 3 we prove Theorem 1.1 for affine permutation groups.  The
main difficulty arising here is that there are linear groups $G\leq
GL(V)$ preserving a direct sum decomposition $V=V_1\oplus\ldots\oplus
V_t$ such that $N_G(V_1)/C_G(V_1)\leq GL(V_1)$ is a large linear
group, while $G$ still acts faithfully on $\{V_1,\ldots,V_t\}$. Therefore, in Section 3.1, we generalise the concept of an imprimitive linear
group in order to be able to use a reduction argument. In Section 3.1 we also consider the case when the
$H$-module $V$ is induced from an $H_1$-module $V_1$ such that
the base size of $H_{1}$ on $V_{1}$ is bounded. In Sections 3.2 and 3.3 we consider two
special cases which we will call alternating-induced and
classical-induced representations. Finally, in Section 3.4 we complete
the proof of Pyber's conjecture for affine permutation groups by using
a structure theorem of Liebeck and Shalev for primitive linear groups
of unbounded base size and by reducing this problem to one of the
previously handled cases in Sections 3.1-3.3. In the final section we indicate that Pyber's conjecture
holds for all non-affine primitive permutation groups with multiplicative
constant $45$.
\section{The distinguishing number of a transitive permutation 
group}\label{sec:dist}

Let $G$ be a group acting (not necessarily faithfully) on a finite set
$\Omega$. A base for $G$ is a subset $\Sigma$ of $\Omega$ such that
the intersection of the stabilizers in $G$ of all points in $\Sigma$
is the kernel of the action of $G$ on $\Omega$. We denote the minimal
size of a base for $G$ by $b(G)$ or by $b_{\Omega}(G)$ if $\Omega$ is
to be specified. More generally, for any normal subgroup $N$ of $G$ we
set $b_{\Omega}(G/N) = \min \{ k\,|\,\exists x_1,\ldots,x_k\in \Omega,
\ \cap_{i=1}^{k} G_{x_i}\leq N \}$. A trivial observation is that
$$\max \{ b(N),b(G/N) \} \leq b(G) \leq b(N) + b(G/N).$$

The purpose of this section is to study yet another invariant which is
closely related to the minimal base size (see Lemma
\ref{rem:dist}).

A distinguishing partition for a finite group $G$ acting (not
necessarily faithfully) on a finite set $\Omega$ is a coloring of the
points of $\Omega$ in such a way that every element of $G$ fixing this
coloring is contained in the kernel of the action of $G$ on
$\Omega$. The minimal number of parts (or colors) of a distinguishing
partition is called the distinguishing number of $G$ and is denoted by
$d(G)$ or by $d_\Omega(G)$.  As for the minimal base size above, for
any normal subgroup $N$ of $G$ we define $d(G/N)$ to be the minimal
number of colors needed to color the points of $\Omega$ in such a way
that the stabilizer in $G$ of this coloring is contained in
$N$. Clearly, for every subgroup $H$ of $G$ and for every normal
subgroup $N$ of $G$ we have $$\max \{d(H),d(G/N) \} \leq d(G) \leq
d(N)d(G/N).$$

The following lemma is of importance to us. 
\begin{lem}\label{rem:dist}
  Let $G$ be a finite group acting on a finite set $\Omega$. For an
  integer $q \geq 2$ let $P^{q}(\Omega)$ denote the set of all
  partitions of $\Omega$ into at most $q$ parts. Then
  $b_{P^{q}(\Omega)}(G) = \left\lceil \log_q(d(G))\right\rceil$.
\end{lem}
\begin{proof}
  Put $\Omega = \{ 1, \ldots, n \}$. We view $P^{q}(\Omega)$ as the
  direct product of $n$ copies of the set $\{ 0, \ldots, q-1
  \}$. Moreover we think of the elements of $P^{q}(\Omega)$ as column
  vectors of length $n$. For a subset $P = \{ v_{1}, \ldots , v_{\ell}
  \}$ of $P^{q}(\Omega)$ let $X$ be the $n$-by-$\ell$ matrix whose
  $\ell$ columns are the vectors $v_{1}, \ldots, v_{\ell}$. Let $D =
  \{ w_{1}, \ldots , w_{n} \}$ be the set of row vectors in $X$. For
  an arbitrary $i$ in $\{ 1, \ldots , n \}$ the vector $w_{i}$ can be
  thought of as the color of the element $i$ in $\Omega$.

  Assume that $D$ does not define a distinguishing partition for $G$
  on $\Omega$. Then there exists an element $g \in G$ that does not
  act trivially on $\Omega$ and preserves the coloring $D$ of
  $\Omega$, that is, $w_{i} = w_{j}$ whenever $i$ is mapped to $j
  (\not= i)$ by $g$. It follows that $g$ fixes every vector in $P$ and
  therefore $P$ is not a base for the action of $G$ on
  $P^{q}(\Omega)$.

  Assume now that $P$ is not a base for the action of $G$ on
  $P^{q}(\Omega)$. Then there exists $g \in G$ fixing every element of
  $P$ such that $g$ does not act trivially on $\Omega$. Since this
  element $g$ preserves the coloring $D$ of $\Omega$, we conclude that
  $D$ is not a distinguishing partition for $G$ on $\Omega$.

  We have shown that the set $P$ is a base for the action of $G$ on
  $P^{q}(\Omega)$ if and only if $D$ defines a distinguishing
  partition (with $|D|$ colors) for $G$ on $\Omega$. The result
  follows.
\end{proof}
The main result (Theorem \ref{d(P)}) of this section determines, up to
an explicit constant factor, the distinguishing number of a transitive
permutation group.

By combining Lemma \ref{rem:dist} and Theorem \ref{d(P)}, we get the
following (almost) equivalent form, a slightly weaker version of which
appears in \cite[Theorem 3.1]{BS}. In the following result, $P(n)$
denotes the power set of $\{1, \ldots, n\}$.
\begin{thm}
\label{equiv}
  For any transitive permutation group $G$ of degree $n > 1$ we have
  $$\frac{\log |G|}{n} < b_{P(n)}(G) < 7 + \frac{\log|G|}{n}.$$
\end{thm}
In the following we aim to prove Theorem \ref{d(P)}. 

Let $\Omega$ be a finite set of size $n > 1$ and $G\leq \sym(\Omega)$
be a (not necessarily transitive) permutation group. 

For the lower bound in the statement of the theorem, notice that the
action of $G$ on $\Omega$ induces an action on the set of all
colorings of $\Omega$ using $d(G)$ colors and this action contains a
regular orbit. Thus $|G| < {d(G)}^{n}$.

From now on we will prove the upper bound in the statement of Theorem
\ref{d(P)}.

Let us first introduce some notation which we will use throughout the
paper. For a finite group $H$ acting on a set $X$ and for a subset $Y$
of $X$, we denote the setwise and the pointwise stabilizer of $Y$ in
$H$ by $N_H(Y)$ and $C_H(Y)$ respectively. In the latter case when $Y
= \{ y_{1}, \ldots , y_{s} \}$ has size $s \geq 1$ we write
$C_{H}(y_{1}, \ldots , y_{s})$. Furthermore, for any natural number
$k$, let $[k]$ denote the set $\{1,2,\ldots, k\}$.

For a system of blocks of imprimitivity for $G$, say $\Gamma=\{ \Delta_1,
\ldots,\Delta_{k} \}$ with $|\Delta_1|=|\Delta_2|=\ldots=|\Delta_k|=m$, 
let $H_j=N_G(\Delta_j)$ for each $j$, and
$N=\cap_{j=1}^k H_j$.  Then $H_j$ acts naturally on $\Delta_j$ with
kernel $C_G(\Delta_j)$, so $H_j/C_G(\Delta_j)\leq \sym(\Delta_j)$.
Furthermore, $G$ acts on $\Gamma$ with kernel $N$, so $K:=G/N\leq
\sym(\Gamma)$.

Our goal is to give an upper bound for the distinguishing number
$d(G)=d_{\Omega}(G)$ of $G$ in terms of the distinguishing numbers
$d(K)=d_\Gamma(K)$ of $K$ and $d(H_j)=d_{\Delta_j}(H_j)$ of $H_j$, and
the degrees $k$ and $m$.

\begin{lem}\label{lem:trivial_bottom_action}
  If $H_j$ acts trivially on $\Delta_j$ (i.e. $H_j=C_G(\Delta_j)$) for
  every $1\leq j \leq k$, then $d(G)\leq
  \lceil\sqrt[m]{d(K)}\rceil$.
\end{lem}
\begin{proof}
  The assumption of the lemma means that each orbit of $G$ on $\Omega$
  has at most one common point with the block $\Delta_j$ for every
  $j\in [k]:=\{1,\ldots,k\}$.  Thus, we can define a function
  $f:\Omega \mapsto [m]$ such that the restriction of $f$ to
  $\Delta_j$ is bijective for every $j$ and $f$ is constant on every
  orbit of $G$. Set $c=\lceil \sqrt[m]{d(K)}\rceil$.
  
  We define a $c$-coloring $\lambda$ of
  $\Omega$ in the following way.  Let us choose a $d(K)$-coloring
  $\alpha:\Gamma\mapsto\{0,1,\ldots,d(K)-1\}$ of $\Gamma$ such that
  only the identity of $K$ fixes $\alpha$.
  For every $j\in [k]$ write 
  $\alpha(\Delta_j)$ in its base $c$-expansion, so 
  \[
  \alpha(\Delta_j)=a_1(j)c^0+a_2(j)c^1+\ldots+a_{s+1}(j) c^s,
  \]
  where $a_1(j),\ldots,a_{s+1}(j)\in\{0,\ldots,c-1\}$. Note that
  $s\leq m-1$ by the definition of $c$. If $s<m-1$, let us define
  $a_{s+2}(j)=\ldots=a_m(j)=0$. Now, for any $x\in\Delta_j$ let
  $\lambda(x)=a_{f(x)}(j)\in\{0,\ldots,c-1\}$.  We claim that only the
  identity element of $G$ preserves $\lambda$.  By assumption, $N=1$,
  so it is enough to show that if $g\in G$ fixes $\lambda$, then $g$
  also fixes $\alpha$. Let $g\in G$ fixing $\lambda$ and
  $g(\Delta_{j})=\Delta_{j'}$ for some $j,j'\in [k]$. Then we have
  $a_{f(x)}(j)=\lambda(x)=\lambda(g(x))=a_{f(g(x))}(j')$ for every
  $x\in \Delta_j$.  Using the properties of $f$, this means that
  $a_{i}(j)=a_i(j')$ for every $i\in [m]$, i.e. $\alpha(\Delta_j)$ and
  $\alpha(\Delta_{j'})$ have the same base $c$-expansion.
\end{proof}
From now on, let us assume that the action of $G$ is transitive (so
$H_{j}/C_{H_{j}}(\Delta_{j}) \leq \sym(\Delta_j)$ are permutation
isomorphic for all $j\in [k]$), and $H_1$ acts on $\Delta_1$ in a
primitive way. For the remainder of this section, we say that the action 
of $H_1$ on $\Delta_1$ is large if $m=|\Delta_1|\geq 5$ and 
$\alt(\Delta_1)\leq H_{1}/C_{H_{1}}(\Delta_{1}) \leq \sym(\Delta_1)$.

\begin{lem}\label{lem:small_bottom}
  With the above notation, if $H_1$ is not large, then 
  $d(G)\leq 4\cdot \lceil\sqrt[m]{d(K)}\rceil$.
\end{lem}
\begin{proof}
  By the results of Seress \cite[Theorem 2]{Seress} and Dolfi
  \cite[Lemma 1]{Dolfi}, $d(H_1)\leq 4$. This means that each
  $\Delta_j$ can be colored with colors $\{0,\ldots,3\}$ such that any
  element of $H_j$ fixing this coloring acts trivially on
  $\Delta_j$. Let $\chi:\Omega\mapsto \{0,\ldots,3\}$ be the union of
  these colorings. Then Lemma \ref{lem:trivial_bottom_action} can be
  applied to the stabilizer of $\chi$ in $G$, so there exist a
  $\lceil\sqrt[m]{d(K)}\rceil$-coloring
  $\lambda:\Omega\mapsto \big\{0,\ldots
  ,\lceil\sqrt[m]{d(K)}\rceil-1\big\}$ such that only the
  identity of $G$ fixes both colorings $\lambda$ and $\chi$. Finally,
  one can encode the pair $(\chi,\lambda)$ by a $4 \cdot
  \lceil\sqrt[m]{d(K)}\rceil$-coloring $\mu$ of $\Omega$ by choosing a
  suitable bijective function, e.g. let $\mu(x)=4\cdot
  \lambda(x)+\chi(x)$.
\end{proof}
It is possible to slightly modify the proof of Lemma
\ref{lem:small_bottom} (still using Lemma
\ref{lem:trivial_bottom_action}) to allow the situation when the
action of $H_{1}$ on $\Delta_{1}$ is not primitive. The modified
statement is the following.
\begin{remark}\label{rem:small_bottom}
  Suppose that $d(H_{1}) \leq c$ for some constant $c$
  where $H_{1}$ does not necessarily act primitively on
  $\Delta_{1}$. Then $d(G)\leq c \cdot
  \lceil\sqrt[m]{d(K)}\rceil$.
\end{remark}
Now we handle the case where the action of $H_1$ is large and $N\neq
1$. Then the socle of $N$ is a subdirect product of alternating groups
$\alt(m)$. More precisely, by \cite[p. 328, Lemma]{Scott}, the socle
of $N$ is of the form $\prod_{j}D_{j}$ where each $D_{j}$ is
isomorphic to $\alt(m)$ and is a diagonal subgroup of a subproduct
$\prod_{\ell \in I_{j}} C_{\ell}$ where $C_{\ell} \cong \alt(m)$ and
the subsets $I_{j}$ form a partition of $\Gamma$ with parts of equal
size. (Moreover, they form a system of blocks for the action of $G$ on
$\Gamma$.) Let us denote the size of each part $I_{j}$ by $t$. In
accordance with \cite{BS}, we will refer to this number as the linking
factor of $N$.  Thus, we have 
\begin{equation}\label{Eq:NequalAltk/t}\tag{Eq.\ 1}
  \alt(m)^{k/t}\leq N\leq\sym(m)^{k/t}.
\end{equation}
\begin{lem}\label{lem:large_nonempty_bottom}
  Let us assume that $H_1$ is large and $N\neq 1$ with
  linking factor $t$. Then $d(G)\leq
  3\cdot\lceil\sqrt[t]m\rceil\cdot \lceil\sqrt[m]{d(K)}\rceil$.
\end{lem}
\begin{proof}
  If $m=6$, then Remark \ref{rem:small_bottom} gives the result. So
  from now on assume that this is not the case. In what follows we
  will prove a slightly stronger inequality in the remaining cases,
  namely $d(G)\leq 2\cdot\lceil\sqrt[t]m\rceil\cdot
  \lceil\sqrt[m]{d(K)}\rceil$.

  Applying suitable bijections $\Gamma\mapsto [k]$ and
  $\Delta_j\mapsto [m]$ for every $j\in [k]$ we can identify $\Omega$
  with $[m]\times[k]=\{(i,j)\,|\,1\leq i\leq m,\ 1\leq j\leq k\}$ such
  that
  \begin{align*}
  N&\leq\{(\sigma_1,\ldots,\sigma_k)\,|\,\sigma_i\in\sym([m]),
  \sigma_a=\sigma_b\textrm{ if }\lceil a/t\rceil =\lceil b/t\rceil\}, \\
  \soc(N)&=\{(\sigma_1,\ldots,\sigma_k)\;\,|\,\;\sigma_i\in\alt([m]),
  \sigma_a=\sigma_b\textrm{ if }\lceil a/t\rceil =\lceil b/t\rceil\},
  \end{align*}
  and the action of any $n=(\sigma_1,\ldots,\sigma_k)\in N$ on
  $[m]\times[k]$ is given as $n(i,j)=(\sigma_j(i),j)$. Under this
  identification, $\Delta_j=\{(i,j)\,|\,i\in[m]\}$ for every
  $j\in[k]$.

  Let $h\in H_j$ for some $j=ut+v\in [k]$ where $v\in[t]$.  Since
  $\soc(N)\nor G$, and the set
  $\{\Delta_{ut+1},\ldots,\Delta_{ut+t}\}$ corresponds to a diagonal
  subgroup of $\soc(N)$, we get that
  $\{\Delta_{ut+1},\ldots,\Delta_{ut+t}\}$ is a block of imprimitivity
  for the action of $G$ on $\Gamma$. Since $H_j$ is by definition the
  stabiliser of $\Delta_{j}$ for some $ut+1\leq j\leq ut+t$, it follows
  that $h\in H_j$ fixes the set 
  \[
  \Omega_u=\Delta_{ut+1}\cup\Delta_{ut+2}\cup\ldots\cup\Delta_{ut+t}
  \] 
  setwise.
  Moreover, since the restriction of $\soc(N)$ to $\Omega_u$ acts 
  on each of $\Delta_{ut+1},\ldots,\Delta_{ut+t}$ 
  in the same way, and the action of $h$ on $\Omega_u$ must  
  normalize this, we get that $h$ acts on $\Omega_u$ coordinatewise 
  i.e. there exist $\sigma_h\in\sym([m]),\ \pi_h\in\sym([t])$ such that
  \[
  h(i,ut+w)=(\sigma_h(i),ut+\pi_h(w))\textrm{ for every }i\in [m],\ w\in[t].
  \]
  
  First let us assume that $t\geq m$.

  We define a $2$-coloring $\chi$ of $\Omega=[m]\times [k]$ as
  \[
  \chi(i,j)=\left\{\begin{array}{ll}
      1&\textrm{if\ \ }i\leq j\hspace{-10pt}\pmod{t} \leq m\\
      0&\textrm{if\ \ }i>j\hspace{-10pt}\pmod{t} 
        \textrm{ or }j\hspace{-10pt}\pmod{t}> m\\
      \end{array}
    \right. .
  \]
  That is, each $\Omega_u$ is colored in the same way; only the first
  $w$ elements of $\Delta_{ut+w}$ are colored with $1$, unless $w>m$
  when no element of $\Delta_{ut+w}$ is colored with $1$. (Notice that
  if $j$ is a multiple of $t$ then here $j\hspace{-3pt}\pmod{t}$ means
  $t$ (not $0$).)
  
  Now, let $h\in H_j$ for some $j=ut+v,\ v=j\hspace{-3pt}\pmod{t}$
  preserving $\chi$. If the action of $h$ on $\Omega_u$ is given by
  $(\sigma_h,\pi_h)\in\sym ([m]) \times \sym ([t])$, then $\sigma_h$ must fix
  each set $[w],\ w\in [m]$, i.e. $\sigma_h=\id_{[m]}$. It follows that
  $h\in H_j$ acts trivially on $\Delta_j$. So, Lemma
  \ref{lem:trivial_bottom_action} can be applied to the stabilizer of
  $\chi$ in $G$ to get a $\lceil\sqrt[m]{d(K)}\rceil$-coloring
  $\lambda$ of $\Omega$ such that only the identity element of $G$
  preserves both $\chi$ and $\lambda$. Finally, as in the last
  paragraph of the previous lemma, the pair $(\chi,\lambda)$ can be 
  encoded with the $2\lceil\sqrt[m]{d(K)}\rceil$-coloring 
  $\mu(x):=2\cdot \lambda(x)+\chi(x)$. 

  Now, let $t<m$. First we define a $2$-coloring $\chi$ of
  $\Omega=[m]\times [k]$ in a similar way as for the previous case:
  \[
  \chi(i,j)=\left\{\begin{array}{ll}
      1&\textrm{if\ \ }i\leq j\hspace{-10pt}\pmod{t}\\
      0&\textrm{if\ \ }i>j\hspace{-10pt}\pmod{t}\\
    \end{array}
  \right. . 
  \]
 
  If $h\in H_j$ for some $j=ut+v,\ v \equiv j\hspace{-3pt}\pmod{t}$
  preserving $\chi$, then $h\in \cap_{w=1}^t H_{ut+w}$ must hold. Moreover, the 
  action of $h$ on each $\Delta_{ut+w}$ must be the same. 

  Second, we can define a $\lceil \sqrt[t] m\rceil$-coloring
  $\beta_u:\Omega_u\mapsto\{0,\ldots,\lceil \sqrt[t] m\rceil-1\}$ for
  every $u$ such that if $h\in H_{ut+v}$ fixes both $\chi$ and
  $\beta_u$, then it acts trivially on $\Omega_u$. This construction
  is analogous to the construction of $\la$ given in the proof of
  Lemma \ref{lem:trivial_bottom_action}. In fact, one can use 
  Lemma \ref{lem:trivial_bottom_action} directly by observing that
  $\{ \Lambda_i=\{(i,ut+w)\,|w\in[t]\} \}_{i}$ is a system of blocks of
  imprimitivity of the stabilizer $T_j$ of $\chi$ in $H_j$ and
  the setwise stabilizer of each $\Lambda_i$ in $T_j$ acts
  trivially on $\Lambda_i$.  Let $\beta:\Omega\mapsto\{0,\ldots,\lceil
  \sqrt[t] m\rceil-1\}$ be the union of the $\beta_u$. 
  Thus, we get that Lemma \ref{lem:trivial_bottom_action} can be applied for 
  the intersections of the stabilizers of $\chi$ and $\beta$. Thus, there 
  is a $\lceil\sqrt[m]{d(K)}\rceil$-coloring $\lambda:\Omega\mapsto
  \big\{0,\ldots,\lceil\sqrt[m]{d(K)}\rceil-1\big\}$ such that only 
  the identity element of $G$ fixes all of the colorings $\chi,\beta,\lambda$. 
  Finally, we can encode the triple $(\chi,\beta,\lambda)$ with
  the $2\cdot\lceil \sqrt[t] m\rceil\cdot \lceil\sqrt[m]{d(K)}\rceil$-coloring
  $\mu$ of $\Omega$ given as
  $\mu(x):=2\cdot\lceil \sqrt[t] m\rceil\lambda(x)+2\cdot \beta(x)+\chi(x)$.
\end{proof}
A permutation group $G \leq \mathrm{Sym}(\Omega)$ is called
quasi-primitive if every non-trivial normal subgroup of $G$ is
transitive on $\Omega$. Clearly, every primitive permutation group is
quasi-primitive.
\begin{lem}\label{quasi}
  If $G \leq \mathrm{Sym}(\Omega)$ is a (finite) quasi-primitive
  permutation group, then $d(G) \leq 4$ or $\mathrm{Alt}(\Omega) \leq
  G \leq \mathrm{Sym}(\Omega)$.
\end{lem}
\begin{proof}
  Let us prove the lemma by induction on $n = |\Omega|$. If $G$ is a
  primitive permutation group, then the claim follows by Seress
  \cite[Theorem 2]{Seress} and Dolfi \cite[Lemma 1]{Dolfi}. Suppose
  that $G$ is not primitive but quasi-primitive. Let $\Gamma$ be a
  system of blocks for $G$ with $k = |\Gamma| < n$ maximal. Let $K
  \cong G$ be the action of $G$ on $\Gamma$. Since a distinguishing
  partition of $\Gamma$ for $K$ gives rise naturally to a
  distinguishing partition of $\Omega$ for $G$, we have $d_{\Omega}(G)
  \leq d_{\Gamma}(K)$. By induction, $d(G) \leq d(K) \leq 4$ or
  $\mathrm{Alt}(\Gamma) \leq K \leq \mathrm{Sym}(\Gamma)$. Thus we may
  assume that $\mathrm{Alt}(k) \leq G \leq \mathrm{Sym}(k)$ with $k
  \geq 5$. Each element of $\Gamma$ is a block of size at least
  $k-1$. For each $i$ with $0 \leq i \leq k-1$ color $i$ letters in
  block $i+1$ with $1$ and the rest $0$. This way we colored the
  elements of $\Omega$ with $2$ colors in such a way that the
  stabilizer in $G$ of this coloring is trivial. Thus $d(G) \leq 2$.
\end{proof}
A permutation group is defined to be innately transitive if there is a
minimal normal subgroup of the group which is transitive. Such groups
were introduced and studied by Bamberg and Praeger \cite{BaPr}. A
quasi-primitive permutation group is innately transitive. The next
theorem is a generalization of Lemma \ref{quasi}. It considers a class
of groups which contains the class of innately transitive groups.
\begin{thm}\label{thm1}
  Let $M \nor G \leq \mathrm{Sym}(\Omega)$ be transitive permutation
  groups where $\Omega$ is finite and $M$ is a direct product of
  isomorphic simple groups. Then $d(G) \leq 12$ or
  $\mathrm{Alt}(\Omega) \leq G \leq \mathrm{Sym}(\Omega)$.
\end{thm}
\begin{proof}
  We prove the claim using induction on $n = |\Omega|$. By Lemma
  \ref{quasi} we may assume that $G$ is not a quasi-primitive
  permutation group.

  As before, let $\Gamma = \{ \Delta_{1}, \ldots, \Delta_{k} \}$ be a
  system of imprimitivity consisting of minimal blocks, each of size
  $m$, for the action of $G$. Let the kernel of the action of $G$ on
  $\Gamma$ be $N$ and set $K = G/N$, a subgroup of
  $\mathrm{Sym}(\Gamma)$.

  We claim that we may assume that $N \not= 1$. Suppose $N = 1$. By
  the induction hypothesis, $d(G) = d_{\Omega}(G) \leq d_{\Gamma}(K)
  \leq 12$, or $G \cong \mathrm{Alt}(\Gamma)$ or $G \cong
  \mathrm{Sym}(\Gamma)$ with $k \geq 13$. In the latter case $G$ is
  quasi-primitive, since $M = \soc(G)$ is transitive. The claim
  follows.

  We claim that we may assume that the action of $H_{1}$ on
  $\Delta_{1}$ is large. For assume that the action of $H_{1}$ on
  $\Delta_{1}$ is not large. By the induction hypothesis, we know that
  $d(K) \leq 12$ or $K$ is an alternating or symmetric group of degree
  at least $13$ in its natural action on $\Gamma$. In the previous
  case the bound $d(G) \leq 12$ follows via Lemma
  \ref{lem:small_bottom} (for $m \geq 3$) and Remark
  \ref{rem:small_bottom} (for $m=2$). Suppose that the latter case
  holds. If $m \geq k-1$, then Lemma \ref{lem:small_bottom} gives
  $d(G) \leq 8$. Suppose that $m < k-1$. Consider the image
  $\overline{M}$ of $M$ under the natural homomorphism from $G$ to $K
  = G/N$. Since $M \triangleleft G$ acts transitively on $\Gamma$, the
  group $\overline{M}$ is a non-trivial normal subgroup of $K$. Thus
  $\overline{M} \cong \mathrm{Alt}(k)$ or $\overline{M} \cong
  \mathrm{Sym}(k)$ with $k \geq 13$. Since $\overline{M}$ is a
  quotient group of $M$ and $M$ is a direct product of isomorphic
  simple groups, $M$ must be a direct product of copies of
  $\mathrm{Alt}(k)$.  Since $m < k-1$, the stabilizer of $\Delta_{1}$
  in $M$ acts trivially on $\Delta_{1}$, and this contradicts the
  transitivity of $M$.
	
  Since the action of $H_{1}$ on $\Delta_{1}$ is non-empty (that is,
  $N \not= 1$) and large, $R = \mathrm{soc}(N)$ is isomorphic to a
  direct product of, say $r$ copies of $\mathrm{Alt}(m)$ where $m \geq
  5$ (see \cite[p. 328, Lemma]{Scott}). Furthermore, since $G$ acts
  transitively on $\Gamma$, the normal subgroup $R$ of $G$ is in fact
  a minimal normal subgroup of $G$.

  We claim that $R \leq M$. Suppose otherwise. Then $R \cap M = 1$
  implies that $R$ is contained in the centralizer $C$ of $M$ in
  $\sym(\Omega)$. Since $M$ is transitive, $C$ must be
  semiregular. However $R$ is not semiregular. Thus $R \leq M$.

  In fact, $R < M$ since $M$ is transitive on $\Gamma$ and $R$ is
  not. Furthermore, since $R$, and thus $M$, is a direct product of copies
  of $\alt(m)$, we must have $k \geq m$. By the fact that $M$ acts
  transitively on $\Gamma$, it also follows that $M$ acts transitively
  on the set of $r$ direct factors of $R$. But every subnormal
  subgroup of $M$ is also normal in $M$, which forces $r=1$ and so 
  the linking factor of $N$ (and also of $R$) is $k$.

  By Lemma \ref{lem:large_nonempty_bottom}, $d(G)\leq
  3\cdot\lceil\sqrt[k]m\rceil\cdot \lceil\sqrt[m]{d(K)}\rceil = 6
  \cdot \lceil\sqrt[m]{d(K)}\rceil$. By the induction hypothesis,
  $d(K) \leq 12$ (in which case $d(G) \leq 12$ by the previous
  inequality) or $K$ is an alternating or a symmetric group of degree
  $k \geq 13$. But in the latter case $m = k$ (and $d(K) \leq m$). Thus
  $\lceil\sqrt[m]{d(K)}\rceil = 2$ and so $d(G) \leq 12$ by Lemma
  \ref{lem:large_nonempty_bottom}.
\end{proof}
\begin{proof}[Proof of Theorem \ref{d(P)}]
  First suppose that $G \leq \mathrm{Sym}(\Omega)$ is a
  quasi-primitive permutation group. By Lemma \ref{quasi}, we may
  assume that $n = |\Omega|\geq 48$ and $\alt(\Omega) \leq G \leq
  \sym(\Omega)$. In this case we have $d(G) \leq n < 48
  \sqrt[n]{n!/2}$ where the second inequality follows from the fact
  that $\frac{1}{2} {(n/3)}^{n} < n!/2$. Thus we may assume that $G
  \leq \mathrm{Sym}(\Omega)$ is not a quasi-primitive permutation
  group.

  Let $M$ be a minimal normal subgroup in $G$ which does not act
  transitively on $\Omega$. Let an orbit of $M$ on $\Omega$ be
  $\Sigma$, and let $\Gamma$ be the set of orbits of $M$ on
  $\Omega$. Let the size of $\Gamma$ be $k$ and let $H$ be the
  stabilizer in $G$ of $\Sigma$. As before, denote the distinguishing
  number of $H$ acting on $\Sigma$ by $d_{\Sigma}(H)$. Since $M \nor
  H$, Theorem \ref{thm1} implies that $d_{\Sigma}(H) \leq 12$ or
  $\alt(\Sigma) \leq H/C_{H}(\Sigma) \leq \sym(\Sigma)$.

  Case 1. $d_{\Sigma}(H) \leq 12$.
 		 
  By Remark \ref{rem:small_bottom}, $d(G)\leq 12 \left\lceil
    \sqrt[m]{d(K)}\right\rceil$ where $K$ is the action of $G$ on
  $\Gamma$ and $m=|\Sigma|$. Since $K$ is a transitive group on $k$
  points, by induction we have $d(K) \leq 48 \sqrt[k]{|K|}$. If $m
  \geq 6$, then
 		 
  $d(G) \leq 12 \left\lceil \sqrt[m]{d(K)}\right\rceil \leq 12
  \left\lceil \sqrt[m]{48 \sqrt[k]{|K|}}\right\rceil \leq 24
  \sqrt[m]{48 \sqrt[k]{|K|}} \leq 48 \sqrt[n]{|K|} \leq 48
  \sqrt[n]{|G|}$.

  If $m \leq 5$ then we can use the previous estimate with $12$
  replaced by $m$ and $24$ replaced by $2m$.

  Case 2. $\mathrm{Alt}(\Sigma) \leq H/C_{H}(\Sigma) \leq
  \mathrm{Sym}(\Sigma)$ with $|\Sigma| = m \geq 13$.
 	
  In this case the action of $H$ on $\Sigma$ is large. Let the kernel
  of the action of $G$ on $\Gamma$ be $N$ and let $t$ be the linking
  factor of $N$. Since $M \leq N$, we know that $N \neq 1$. Set
  $\epsilon = 1$ if $t=1$ and $\epsilon = 2$ if $t \not= 1$. Then
  Lemma \ref{lem:large_nonempty_bottom} implies that $$d(G) \leq
  3\lceil\sqrt[t]m\rceil \lceil\sqrt[m]{d(K)}\rceil \leq 6 \epsilon
  \sqrt[t]m\sqrt[m]{d(K)} = 6 \epsilon
  \sqrt[mk]{m^{mk/t}}\sqrt[m]{d(K)}.$$ Set $c = 6 \cdot 2^{1/mt} \cdot
  3^{1/t}$. By use of the inequality $\frac{1}{2} {(m/3)}^{m} < m!/2 =
  |\alt(m)|$, we have that $d(G)$ is at most
  $$6 \epsilon \sqrt[mk]{m^{mk/t}}\sqrt[m]{d(K)} < 
  6 \epsilon \sqrt[mk]{{((m!/2) \cdot 2 \cdot
      3^{m})}^{k/t}}\sqrt[m]{d(K)} \leq c \cdot \epsilon
  \sqrt[n]{{(|\alt(m)|)}^{k/t}}\sqrt[m]{d(K)}.$$ As noted in
  (\ref{Eq:NequalAltk/t}), we have that ${\alt(m)}^{k/t} \leq N$. This
  gives the inequality $d(G) < c \cdot \epsilon \sqrt[n]{|N|}
  \sqrt[m]{d(K)}$. By the induction hypothesis, we have $d(K) \leq 48
  \sqrt[k]{|K|}$. Thus $$d(G) < c \cdot \epsilon \sqrt[m]{48}
  \sqrt[n]{|N|} \sqrt[n]{|K|} \leq 6 \cdot \epsilon \cdot 2^{1/13t}
  3^{1/t} \sqrt[13]{48} \sqrt[n]{|G|} < 48 \sqrt[n]{|G|}.$$
\end{proof} 				
\section{The affine case}

\subsection{Some reductions and notation}\label{sec:reduc}

We begin our study of Theorem \ref{Pyber} in the case of affine
primitive permutation groups.

Let $G$ be an affine primitive permutation group acting on a finite
set $\Omega$. Then $G$ contains a unique minimal normal subgroup $V$
acting regularly on $\Omega$, so $|\Omega|=p^d$ for some prime $p$ and
it can be identified with the finite vector space $V$ over $\FF p$ of
dimension $d$. Furthermore, $G=V\rtimes H$ for some $H\leq GL(V)$ and
$H$ acts faithfully and irreducibly on the vector space $V$. Clearly,
$b(G) = b_V(G)= b_{V}(H) + 1$.

In this section we will show that there exists a universal constant $c
>0$ such that for the affine primitive permutation group $G = V\rtimes
H$, we have \[b_{V}(H) \leq 45 (\log |H| / \log |V|) + c.\]

The following theorem shows that we may assume that $H$ acts
imprimitively (and irreducibly) on $V$.
\begin{thm}[Liebeck, Shalev \cite{LS1}, \cite{LS2}]\label{LieSha}
  There exists a universal constant $c>0$ such that if $H$ acts
  primitively on $V$, then $b_{V}(H) \leq \max \{ 18 (\log |H| /
  \log |V|) + 30 \ , \ c \}$.
\end{thm}
Thus we may assume that $V$ is an imprimitive irreducible $\FF p
H$-module. Let $V=\oplus_{i=1}^t V_i$ be a decomposition of $V$ into a
sum of subspaces $V_{i}$ of $V$ that is preserved by the action of
$H$. 
For every $i$ with $1\leq i\leq t$, let $H_i=N_{H}(V_i)$ and let
$K_i=H_i/C_{H_i}(V_i)\leq GL(V_i)$ be the image of the restriction of
$H_i$ to $V_i$.  The group $H$ acts transitively on the set
$\Pi=\{V_1,\ldots,V_t\}$. Let $N$ be the kernel of
this action and let $P$ be the image of $H$ in $\sym(\Pi)$. So
$N=\cap_{i=1}^{t} H_i$ and $P=H/N$.

As an easy application of the results of Section \ref{sec:dist}, we
first prove Theorem \ref{Pyber} in the case when each $b_{V_i}(K_i)$
is bounded (see Theorem \ref{thm:boundedbK_1}). Note that because
the action of $P$ on $\Pi$ is transitive, it is enough to assume
this only for $K_1$. First we handle the even more special case when
$K_1$ is trivial.
\begin{lem}\label{lem:trivK_1}
  If $K_1=1$, then $b_{V}(H)=\lceil\log_{|V_1|}d_\Pi(P)\rceil$.
\end{lem}
\begin{proof}
  First note that the condition $K_1=1$ implies that every orbit of
  $H$ in $\cup_{i=1}^t V_i$ contains exactly one element from every
  subspace $V_i$, which defines a one-to-one
  correspondence $\alpha_{ij}:V_i\mapsto V_j$ 
  between any pair of subspaces $V_i$ and $V_j$.

  Let $b$ be a positive integer. Let
  $w_s=v_s^{(1)}+v_s^{(2)}+\ldots+v_s^{(t)}$ be vectors in $V$ for
  $1\leq s\leq b$ decomposed with respect to the direct sum
  decomposition $V=\oplus_i V_i$. We define an equivalence relation on
  $\Pi$ by $V_i\sim V_j$ if and only if $(v_1^{(i)},\ldots,v_b^{(i)})$
  corresponds to $(v_1^{(j)},\ldots,v_b^{(j)})$,
  i.e. $\alpha_{ij}(v_s^{(i)})=v_s^{(j)}$ for every $1\leq s\leq b$.
  Then the set $\{w_1,\ldots, w_b\}$ is a base for $H$ on $V$ if and
  only if $\sim$ defines a distinguishing partition for $P$ on
  $\Pi$. The number of different vectors of the form $(v^{(i)}_{1},
  \ldots, v^{(i)}_{b})$ with entries from $V_{i}$ (for any $i$) is
  ${|V_{1}|}^b$. It follows that $b_{V}(H)$ is the smallest integer
  such that ${|V_{1}|}^{b_{V}(H)}$ is at least $d_{\Pi}(P)$.
\end{proof}
\begin{remark}
\label{remark1}
  Note that this proof also works if $P$ is not 
  transitive on $\Pi$ but $K_i=1$ for every $i$ with $1\leq i\leq t$.
\end{remark}
\begin{thm}\label{thm:boundedbK_1}
  Let us assume that 
  $b_{V_1}(K_1)\leq b$ for some constant $b$. Then 
  we have 
  \[
  b_V(H)\leq b+1+ \log 48 +  \frac{\log |P|}{\log |V|}.
  \]
\end{thm}
\begin{proof}
  By our assumption, for each $1\leq i\leq t$ we can choose a base
  $\{v_1^{(i)},v_2^{(i)},\ldots,v_b^{(i)}\}\subset V_i$ for $K_i\simeq
  H_i/C_{H_i}(V_i)$. Put $w_s=\sum_{i=1}^tv_s^{(i)}$ for every $1\leq
  s\leq b$ and let $L=\cap_s C_{H}(w_s)$. Then $L \cap H_{i} = C_L(V_i)$
  for every $i$ so we can apply Lemma \ref{lem:trivK_1} for $L$ (see
  also Remark \ref{remark1}). Hence $b_{V}(H)\leq
  b+\lceil\log_{|V_1|}d_\Pi(P)\rceil$. Since $d_\Pi(P)\leq 48
  \sqrt[t]{|P|}$ by Theorem \ref{d(P)}, we get 
  \[b_{V}(H)\leq b+1+\log_{|V_1|}(48 \sqrt[t]{|P|}) \leq b+1+ \log 48 +
  \frac{\log |P|}{t\log |V_1|} =b+1+ \log 48 +\frac{\log|P|}{\log
    |V|},\]
  as claimed.
\end{proof}
Note that Theorem \ref{thm:boundedbK_1} proves Theorem \ref{Pyber} in
case $b +1 + \log 48$ is bounded. In other
words, we must now look at situations when $b_{V_1}(K_1)$ is not
bounded by any fixed constant.

For the remainder of this section, it will be more convenient for us to
use the language of group representations. So,
instead of choosing $H$ as a fixed linear subgroup of $GL(V)$, 
let $H$ be a fixed abstract group and $X:H\ra GL(V)$ a representation 
of $H$. Then we would like to give an upper bound for 
$b_V(X(H))$. (The 
reason for this is that in the proof, we will reduce this problem
to some other representations of $H$ with simpler image structure.)  
Moreover, in order to use a theorem of Liebeck and Shalev 
\cite[Theorem 1]{LS2}, we may also need to extend the base field to 
consider vector spaces over $\FF q$ for some $p$-power $q$. (Of
course, the base size $b_V(X(H))$ is independent on whether we view $V$
as an $\FF p$-space or as an $\FF q$-space.) Occasionally, we 
want to view the vector space $V$ over $\FF q$ as a vector space over $\FF p$, 
which we will emphasize by the notation $V(p)$.

By using our previous notation, we assume that $V=\oplus_{i=1}^t V_i$
is a direct sum of $\FF q$-spaces and $X: H\ra GL(V)$ is a
representation such that $X(H)$ permutes the set
$\Pi=\{V_1,\ldots,V_t\}$ in a transitive way.  Thus, the representation
$X$ is equivalent to the induced representation $\ind_{H_1}^H(X_1)$, where
$X_1:H_1\ra GL(V_1)$ is a linear representation of $H_1$. 

In Sections \ref{sec:alt}
and \ref{sec:classical} we first consider two special cases, which we
will respectively call alternating-induced and classical-induced
classes. Here alternating-induced means that $K_1$ is isomorphic to an
alternating or symmetric group, and $V_1$ as an $\FF q K_1$-module is
the deleted permutation module for $K_1$.  Similarly,
classical-induced means that $K_1$ is a classical group (maybe over
some subfield $\FF {q_0}\leq \FF q$) with its natural action on $V_1$.
Then we show in Section \ref{sec:elim} how the general case can be
reduced to one of these modules. 

In fact, in order to be able to use a reduction argument in Section
\ref{sec:elim}, we need to work with the following natural
generalization of projective representations.
\begin{defin}\label{def:T-mod}
  Let $V$ be a finite vector space over $\FF q$ and $T\leq GL(V)$ any subgroup. 
  We say that a map $X:H\ra GL(V)$ is a $\pmod T$-representation of $H$ if
  the following two properties hold:
  \begin{enumerate}
    \item $X(g)$ normalizes $T$ for every $g\in H$;
    \item $X(gh)T=X(g)X(h)T$ for every $g,h\in H$.
  \end{enumerate}
\end{defin}

\begin{defin}\label{def:T-mod_equi}
  Let $T\leq GL(V)$ and $X_1,X_2:H\ra GL(V)$ be two $\pmod
  T$-representa\-ti\-ons of $H$.  We say that $X_1$ and $X_2$ are $\pmod
  T$-equivalent if there is an $f\in N_{GL(V)}(T)$ such that
  $X_1(g)T=f X_2(g)f^{-1}T$ for all $g\in G$.
\end{defin}
For a $\pmod T$-representation $X:H\ra GL(V)$, we define the
corresponding base size of $H$ as 
\begin{equation}\label{Eq:bXH_def}\tag{Eq.\ 2}
  b_{X}(H):=b_V(X(H)T) 
\end{equation}
(note that
$X(H)T$ is a subgroup of $GL(V)$). It is easy to see that equivalent
$\pmod T$-representations have the same base size. Note that $b_{V}(H)
\leq b_{X}(H)$ in case $H \leq GL(V)$ and $X = \mathrm{id}$.

For $T = 1$ a $\pmod T$-representation is the same as a linear representation. 

In this paragraph let $T=Z(GL(V)) \simeq \FF q^\times$ be the group of
all scalar transformations on $V$. Then a $\pmod{T}$-representation of
$H$ is the same as a projective representation of $H$. Furthermore, in
this case $T$-equivalence of two $T$-representations of $H$ means
exactly that they are projectively equivalent.  Slightly more
generally, if $X:H\rightarrow GL(V(p))$ is any map satisfying (1) of
Definition \ref{def:T-mod} (still with the assumption that $V$ is an 
$\FF q$-space and $T\simeq \FF q^\times$), then $X(h)$ acts on $T$ by a field
automorphism $\sigma(h)\in \aut(\FF q)$ for any $h\in H$, so $X(H)$ is
contained in the semilinear group $\Gamma L(V)=GL(V)\rtimes \aut(\FF
q)$. In the following, we will also call such a
$\pmod{T}$-representation $X:H\rightarrow \Gamma L(V(p))$ a projective
representation.  Furthermore, for any projective
representation $X:H\ra \Gamma L(V)$, we will denote by $\fX$ the
associated homomorphism $H\ra P\Gamma L(V)$ (which we again call a
projective representation).

For the remainder, we consider the special case where 
$V=\oplus_{i=1}^t V_i$ is a direct sum of $\FF q$-spaces, and 
\begin{equation}\label{Eq:TVdef}\tag{Eq.\ 3}
  T_V=\{g\in GL(V)\,|\,g(V_i)=V_i \textrm{ and }g|_{V_i}\in Z(GL(V_i)) \ \
  \forall 1\leq i\leq t\}\simeq (\FF q^\times)^t.
\end{equation}

If a direct sum decomposition of a vector space 
$U$ is given, then $T_U$ will always denote the appropriate subgroup 
defined by the above displayed formula. 

If $q>2$ and $X:H\ra GL(V)$ is an arbitrary map, then $X$ satisfies
(1) of Definition \ref{def:T-mod} (with $T=T_V$) if and only if the
direct sum decomposition $V=\oplus_{i=1}^t V_i$ is preserved by
$X(H)$. In particular, if $X$ happens to be a linear representation of
$H$ preserving the direct sum decomposition $V=\oplus_{i=1}^t V_i$,
then $X$ is also a $\pmod {T_V}$-representation of $H$.

A further observation is that if $X:H\ra GL(V(p))$ is a $\pmod
{T_V}$-representation, then the restricted map $X_{i}:H_i\ra GL(V_i)$ is
a projective representation of $H_i$. (Here $X_{i}$ is defined so
that first we take the restriction of $X$ to $H_i$, then we restrict
the action of $X(H_i)$ to $V_i$.) Conversely, if $X_1:H_1\ra \Gamma
L(V_1)$ is any projective representation, then the induced
representation $X = \ind_{H_1}^H(X_1):H\ra GL(V(p))$ will be a $\pmod
{T_V}$-representation of $H$ transitively permuting the $V_i$, and it
is easy to see that every $\pmod {T_V}$-representation of $H$ transitively
permuting the $V_i$ can be obtained in this way. Here the induced
representation $X=\ind_{H_1}^H(X_1)$ can be defined with the help of a
transversal in $H$ to $H_1$, so it is not uniquely defined. However,
it is uniquely defined up to $\pmod {T_V}$-equivalence, so this will not
be a problem for us.

So, for the remainder, we assume that the groups $H_1\leq H$ are fixed, 
and we consider representations of the form $X=\ind_{H_1}^H(X_1)$, where 
$X_1:H_1\ra \Gamma L(V_1)$ is a projective representation of $H_1$. 

\subsection{Alternating-induced representations}\label{sec:alt}
In this subsection we will only consider linear representations 
$X:H\ra GL(V)$ and $X_i:H_i\ra GL(V_i)$ such that 
$X=\ind_{H_i}^H(X_i)$ for all $i$. 
We also assume that for all $i$ with $1 \leq i \leq t$, the
groups $K_i=X_i(H_i)\leq GL(V_i)$ are isomorphic to some alternating or
symmetric group of degree $k$ at least $7$, and $K_i$ acts on $V_i$
such that $V_i$ as an $\FF qK_i$-module ($q$ is a power of $p$) is
isomorphic to the non-trivial irreducible component of the permutation
module obtained from the natural permutation action of $K_i$ on a fixed
basis of a vector space of dimension $k$ over $\FF q$. In this
situation we say that $V\simeq \ind_{H_1}^H(V_1)$ is an alternating-induced 
$\FF q H$-module, and $X:H\ra GL(V)$ is an alternating-induced representation.  

In the following proposition we describe the construction of the module $V_{i}$.
\begin{prop}\label{proposition}
Let $K\simeq \alt(k)$ or $\sym(k)$ and consider its action on an $\FF q$ 
vector space $U$ of dimension $k \geq 5$, defined by permuting the elements 
of a fixed basis $\{ e_1,\ldots,e_k \}$ of $U$. Let us define the subspaces
\[
U_0=\Big\{\sum_i \alpha_ie_i\,|\,\alpha_i\in\FF q,\ \sum_i\alpha_i=0\Big\}
\qquad\textrm{and}\qquad
W=\Big\{\alpha(\sum_i e_i)\,|\,\alpha\in\FF q\Big\}.
\]
\begin{enumerate}
\item If $p\nmid k$, then $U=U_0\oplus W$, $W$ is isomorphic to the trivial 
  $\FF qK$-module and $U_0$ is the unique non-trivial 
  irreducible component of the $\FF qK$-module $U$. 
\item If $p\mid k$, then $U\geq U_0\geq W$, both $U/U_0$ and $W$ are
  isomorphic to the trivial $\FF q K$-module and $U_0/W$ is the unique
  non-trivial irreducible component of the $\FF qK$-module $U$.
\end{enumerate}
\end{prop} 
\begin{proof}
This is well known (see \cite[Page 185]{KL}, for example).
\end{proof}
We can apply Proposition \ref{proposition} to each pair $K_i$, $V_i$
to define $\FF qK_i$-modules $U_i$ and their submodules $U_{i,0},\
W_i\leq U_i$ such that either $V_i\simeq U_{i,0}$ (for $p\nmid k$) or
$V_i\simeq U_{i,0}/W_i$ (for $p\mid k)$. Then the original action of
$H$ on $V$ may be defined using the action of $H$ on $U:=\oplus_i
U_i$. Moreover, if we choose a basis $\{ e_1^{(i)},\ldots,e_k^{(i)} \}
\subset U_i$ for every $i$ as in Proposition \ref{proposition} in a
suitable way, then $\{e_j^{(i)}\,|\,1\leq i\leq t,\ 1\leq j\leq k\}$
will be a basis of $U$ such that $H$ acts on $U$ by permuting the
elements of this basis.

The next lemma says that $b_V(H)$ is bounded by a linear function of $b_U(H)$. 

\begin{lem}\label{lem:altmain}
With the above notation $b_V(H)\leq 2b_U(H)+3$ for $k\geq 7$.
\end{lem}
\begin{proof}
  First, we define three vectors $w_1,w_2,w_3\in U_{1,0}\oplus
  U_{2,0}\oplus\ldots \oplus U_{t,0}$ as linear combinations of the basis
  vectors $\{e_j^{(i)}\,|\,1\leq i\leq t,\ 1\leq j\leq k\}$ as
  follows.
  \[
  w_{1} = \sum_{i=1}^t(e_1^{(i)}-e_2^{(i)}),\ w_{2} = \sum_{i=1}^t(e_2^{(i)}-e_3^{(i)}),\ 
 w_{3} = \sum_{i=1}^t(e_3^{(i)}-e_4^{(i)}).
  \]
  Let $L=C_{H}(  w_1,w_2,w_3 )$, so $\{e_j^{(i)}\,|\,1\leq i\leq t\}$
  are $L$-invariant subsets for $1\leq j\leq 4$.  

  Let $\{u_1,\ldots,u_b\}\subset U$ be a base for $H$ of size
  $b=b_U(H)$.
  Now, for any $u\in \{u_1,\ldots,u_b\}$ we define two further vectors 
  $u^e,u^f\in U_{1,0}\oplus U_{2,0}\oplus\ldots \oplus U_{t,0}$ as follows. 
  Write $u=\sum_{i,j} a_{ij}e_j^{(i)}$ and define 
  \begin{align*}
    u^e&= \sum_{i} \sum_{j>2} a_{ij}e_j^{(i)}+\sum_i \beta_ie_1^{(i)},
         \textrm{ for } \beta_i=-\sum_{j>2}a_{ij},\\
    u^f&= \sum_{i} \sum_{j\leq 2} a_{ij}e_j^{(i)}+\sum_i \gamma_ie_3^{(i)},
         \textrm{ for } \gamma_i=-(a_{i1}+a_{i2}).\\
  \end{align*}   
  The above definition of the $\beta_i$ and $\gamma_i$ ensures
  that the projection of $u^e$ and $u^f$ to any $U_i$ is really in
  $U_{i,0}$.  Furthermore, if $l\in L$ fixes $u^e$, then because of
  the above mentioned $L$-invariant subsets of basis vectors we get
  that $l$ must fix both $\sum_i \beta_ie_1^{(i)}$ and $\sum_{i} \sum_{j>2}
  a_{ij}e_j^{(i)}$. Similarly, if $l\in L$ fixes $u^f$ then it must
  fix both $\sum_i \gamma_ie_3^{(i)}$ and $\sum_{i} \sum_{j\leq 2}
  a_{ij}e_j^{(i)}$. As a consequence every element of $C_{L}( u^{e},u^{f} )$
  must also fix $\sum_{i} \sum_{j>2} a_{ij}e_j^{(i)}+ \sum_{i} \sum_{j\leq 2}
  a_{ij}e_j^{(i)}=u$. Applying this construction to $u_1,\ldots,u_b$ we get that 
  \[\{w_1,w_2,w_3,u_1^e,u_1^f,u_2^e,u_2^f,\ldots,u_b^e,u_b^f\}\]
  is a base of size $2b+3$ for $H$ acting on $U_{1,0}\oplus\ldots\oplus U_{t,0}$.

  If $p\nmid k$, then there is nothing more to do, since in this case
  $V\simeq U_{1,0}\oplus\ldots\oplus U_{t,0}$ as $\FF q H$-modules.

  For the remainder, let $p\mid k$ and $W=W_1\oplus\ldots\oplus W_t$
  where $W_i$ is the $1$-dimensional submodule of $U_{i,0}$ for all $i$
  with $1 \leq i \leq t$.  For any $x\in U$, let $\ol x=x+W\in U/W$ be
  the associated element in the factor space.  Now, we claim that
  \[\{\ol w_1,\ol w_2,\ol w_3,\ol u_1^e,\ol u_1^f,
  \ol u_2^e,\ol u_2^f,\ldots,\ol u_b^e,\ol u_b^f\}\]
  is a base for $H$ acting on $(\oplus_i U_{i,0})/W\simeq V$. 

  Let $z_i=\sum_j e_j^{(i)}$ for every $1\leq i\leq t$, so
  $\{z_1,\ldots,z_t\}$ is a basis for $W$.  An element $g\in H$ fixes
  $\ol w_s$ (where $s\in\{1,2,3\}$) if and only if there are field
  elements $\lambda_1,\ldots,\lambda_t$ such that
  $g(w_s)=w_s+\sum_i\lambda_i z_i$.  But $g$ permutes the basis
  vectors in $\{e_j^{(i)}\,|\,1\leq i\leq t,\ 1\leq j\leq k\}$ and also the
  subspaces $\{U_{i,0}\,|\,1\leq i\leq t\}$. A consequence of this is
  that the projection of $g(w_s)$ to
  any $U_{i,0}$ must be a non-zero linear combination of exactly two
  basis vectors from $\{e_j^{(i)}\,|\,1\leq j\leq k\}$. Since $k\geq
  7$, this can happen only if $\lambda_i=0$ for every $1\leq i\leq t$,
  i.e. when $g$ fixes $w_s$. So $C_{H}(\ol w_s)=
  C_{H}(w_s)$ for every $s$ with $1\leq s\leq 3$. 
  The same argument can be applied to prove that $C_{H}(\ol u_s^f)=C_{H}(u_s^f)$ 
  for every $1\leq s\leq b$. 

  Finally, let us assume that $g\in C_{H}(\ol w_1,\ol w_2,\ol w_3 )=L$ and 
  $g(\ol u_s^e)=\ol u_s^e$ for some $1\leq s\leq b$. Again this means that 
  $g(u_s^e)=u_s^e+\sum_i\lambda_i z_i$ for some field elements 
  $\lambda_1,\ldots,\lambda_t$. But the linear combination we used to define 
  $u_s^e$ contains no $e_2^{(i)}$ with non-zero coefficient. In other words 
  $u_s^e$ is contained in the $L$-invariant subspace 
  generated by $\{e_j^{(i)}\,|\,j\neq 2,\ 1\leq i\leq t\}$, so this must also 
  hold for $g(u_s^e)=u_s^e+\sum_i\lambda_i z_i$, which implies that 
  $\lambda_i=0$ for every $i$, i.e. $C_L(\ol u_s^e)=C_L(u_s^e)$ holds. 
  We proved that 
  \[C_{H}( \ol w_1,\ol w_2,\ol w_3,\ol u_1^e,\ol u_1^f,
  \ldots,\ol u_b^e,\ol u_b^f ) = C_{H}(  w_1,w_2,w_3,u_1^e,u_1^f,
  \ldots,u_b^e,u_b^f )=1,\]
  as claimed.
\end{proof}
We can now establish Theorem \ref{Pyber} for alternating-induced groups.
\begin{thm}
\label{alternating}
  If $H \leq GL(V)$ is an alternating-induced linear group, then 
  $$b_V(H) \leq 17 + 2 \frac{\log |H|}{\log |V|}.$$
\end{thm}
\begin{proof}
  By definition, $k\geq 7$.  By using the same notation as
  above let $H$ act on $U$ by permuting the basis
  $B=\{e_j^{(i)}\,|\,1\leq i\leq t,\ 1\leq j\leq k\}$. This action is
  clearly transitive, so we can use Theorem \ref{d(P)} to conclude
  that we can color the basis vectors by using at most
  $48\sqrt[kt]{|H|}$ colors such that only the identity of $H$ fixes
  this coloring, i.e.  $d_B(H)\leq 48\sqrt[kt]{|H|}$. Now any vector
  $u\in U$ can be seen as a coloring of this basis by using at most
  $|\FF q|=q$ colors.  By Lemma \ref{rem:dist}, it follows
  that
  \[b_U(H)\leq \lceil \log_q (d_B(H)) \rceil \leq 
  \lceil \log_q (48\sqrt[kt]{|H|}) \rceil <
  7+\frac{\log |H|}{kt\log{q}}=7+\frac{\log|H|}{\log |U|}.\]
  By Lemma \ref{lem:altmain}, $b_V(H)\leq 2b_U(H)+3\leq 17 + 2
  (\log|H|/\log|V|)$, as claimed. 
\end{proof}
\subsection{Classical-induced representations without multiplicities}
\label{sec:classical}

In this subsection let $q$ be a power of the prime $p$,
$V=\oplus_{i=1}^t V_i$ be a direct sum of $\FF q$ vector spaces, and
define $T_V$ as in (\ref{Eq:TVdef}). Let $k$ denote the $\FF
q$-dimension of each $V_i$. Throughout this subsection we will assume
that $k\geq 9$ holds.  We also use the notation $H_i,\Pi,N$ defined in
Section \ref {sec:reduc}.

Let $X:H\ra GL(V(p))$ be a $\pmod {T_V}$-representation of $H$ such
that $X(H)T_V$ acts on $\Pi=\{V_1,\ldots,V_t\}$ in a transitive way. By
our discussion at the end of Section \ref {sec:reduc}, this means that
$X=\ind_{H_i}^H (X_i)$, where $X_i:H_i\ra \Gamma L(V_i)$ is a
projective representation of $H_i$ for every $1\leq i\leq t$. Then
there is an associated homomorphism $\fX:H\ra N_{GL(V(p))}(T_V)/T_V$
defined by $\fX(h):=X(h)T_V/T_V$. For the remainder of this subsection
let $L=\fX(H)$ be the image of this homomorphism. Note that the action
of $H$ on $\Pi$ induces an action of $L$ on $\Pi$.

In this subsection we additionally assume that $X$ is
classical-induced, i.e. for each $i$, the image $K_i$ of the
homomorphism $\fX_i:H_i\rightarrow P\Gamma L(V_i)$ is some classical
group i.e. $S_i=\soc(K_i)\leq P\Gamma L(V_i)$ is isomorphic to some
simple classical group $S$ over some subfield $\FF {q_0}$ of $\FF
q$. Because of our assumption $k \geq 9$, the group generated by all
inner, diagonal and field automorphisms of $S$ (for the remainder, we
denote this group by $\textrm{IDF}(S)$) has index at most $2$ in $\aut
(S)$.

We introduce some further notation. For an $H$-block $\Delta\subseteq
\Pi$ let $V_\Delta:=\oplus_{V_i\in\Delta}V_i$, and
$X_\Delta:N_H(\Delta)\ra GL(V_\Delta(p))$ be the $\pmod
{T_{V_\Delta}}$-representation of $N_H(\Delta)$ defined by taking the
restriction of $X(h)$ to $V_\Delta$ for all $h\in N_H(\Delta)$. In
particular, $X_\Pi=X$ and $X_{\{V_i\}}=X_i$ holds for each $V_i\in
\Pi$.  Furthermore, let the associated homomorphism $\fX_\Delta$ be
$\fX_\Delta(h):=X_\Delta(h)T_{V_\Delta}/T_{V_\Delta}$.  Define
$L_\Delta=\fX_\Delta(N_H(\Delta))$ and
$S_\Delta:=\soc(\fX_\Delta(C_H(\Delta)))\nor L_\Delta$. If
$\fX_{\Delta}(C_{H}(\Delta)) = 1$, then we set $S_{\Delta} = 1$.
Finally, let $\widetilde{S_\Delta}\leq N_H(\Delta)$ be the inverse
image of $S_\Delta$ under the function $\fX_\Delta$. Then $\fX_i$ is
defined on $\widetilde{S_\Delta}$ for each $V_i\in \Delta$ and it
induces a homomorphism on $S_\Delta$, which we also denote by
$\fX_i:S_\Delta\ra P\Gamma L(V_i)$.

We next introduce a condition which we will additionally assume in
this subsection. 

\begin{defin}[Multiplicity-free condition]
\label{def:MultFree} 
If $\Delta\subseteq \Pi$ is an $H$-block such that $S_\Delta\simeq S$
  and all $\fX_i:S_\Delta\ra P\Gamma L(V_i)$ for $V_i\in\Delta$ are
  projectively equivalent, then $|\Delta| =1$. 
\end{defin}

A consequence of this assumption is the following.
\begin{prop}\label{prop:linkingfact}
  Suppose $X$ is classically-induced and let $\Delta\subseteq \Pi$ be an
  $H$-block such that $S_\Delta\simeq S$. If the
  multiplicity-free condition holds, then $|\Delta|\leq 2$.
\end{prop}
\begin{proof}
  First note that if $\Delta'\subset \Delta$ is any $H$-block, then
  the assumption $S_\Delta\simeq S$ implies that $S_{\Delta'}\simeq
  S$.  For simpler notation, we can assume that
  $\Delta=\{V_1,\ldots,V_d\}$ for $d=|\Delta|$. By assumption,
  $S_\Delta$ is a diagonal subgroup of $S_1\times\ldots\times
  S_d\simeq S^d$.  So, $S_\Delta$ can be identified with
  $\{(s,s^{z_2},\ldots,s^{z_d})\,|\,s\in S\}$, where
  $z_2,\ldots,z_d\in\aut(S)$ are fixed elements. Now, if
  $z_i^{-1}z_j\in\textrm{IDF}(S)$, then $\fX_i:S_\Delta\ra P\Gamma
  L(V_i)$ and $\fX_j:S_\Delta\ra P\Gamma L(V_j)$ are projectively
  equivalent. The relation $V_i\sim V_j\iff z_i^{-1}z_j\in
  \textrm{IDF}(S)$ defines an $N_H(\Delta)$-congruence on $\Delta$.
  Using that $|\aut(S):\textrm{IDF}(S)|\leq 2$ and the first sentence
  of the proof, we get that there is an $H$-block $\Delta'\subset
  \Delta$ such that $|\Delta'|\geq |\Delta|/2$, $S_{\Delta'}\simeq S$
  and all $\fX_i:S_{\Delta'}\ra P\Gamma L(V_i)$ for $V_i\in\Delta'$
  are projectively equivalent. Thus, the result follows from
  the multiplicity-free condition.
\end{proof}
For the rest of this subsection let $\Delta \subseteq \Pi$ be an
$H$-block. The group $S_\Delta$ is either trivial or is a subdirect
product of isomorphic simple classical groups. As for
subdirect products of alternating groups in Section \ref{sec:dist},
this means that $S_\Delta$ is a direct product of diagonal subgroups
corresponding to a partition $\Delta =\cup_i \Delta_i$ of $\Delta$ into
equal-size parts.  Again, we call the size of the parts of this partition
the linking factor of $S_\Delta$.  Note that the $\Delta_i$
themselves are $H$-blocks and $S_{\Delta_i}\simeq S$ for each
$i$. Hence, by Proposition \ref{prop:linkingfact}, the linking factor
of $S_\Delta$ is at most $2$. As before, let $N=C_H(\Pi)$ be the
kernel of the action of $H$ on $\Pi$.

Recall the definitions of $X_{1}$ and $K_{1}$ from the second and
third paragraphs of Section \ref{sec:classical}. The base size
$b_{X_{1}}(K_{1})$ is defined as in (\ref{Eq:bXH_def}). With this
notation the following result is a consequence of Theorem
\ref{LieSha}.
\begin{thm}\label{LiebeckShalev}
  With the above assumptions, there exists a universal constant $c > 0$
  such that $b_{X_1}(K_1) \leq 18 (\log |K_1|)/(\log |V_{1}|) + c$.
\end{thm}  
We can now prove Theorem \ref{Pyber} for such classical-induced
representations which have the multiplicity-free condition.
\begin{thm}\label{classical}
  There exists a universal constant $c > 0$ such that if $X:H\ra
  GL(V)$ is a $\pmod{T_V}$-representation of $H$ (with respect to some
  direct sum decomposition $V=\oplus_{i=1}^t V_i$), which is a
  classical-induced representation possessing the multiplicity-free
  condition, then $b_{X}(H) \leq 45 (\log |H|)/(\log |V|) + c$.
\end{thm}
\begin{proof}
  Assume that $\fX(N) \neq 1$. Then $\soc(\fX(N))=S_\Pi$ for the
  $H$-block $\Pi$, so $\soc(\fX(N))$ is a subdirect product of the
  simple classical groups $S_i$ with linking factor at most $2$. Thus
  $|N| \geq {|S_{1}|}^{t/2} \geq {|K_{1}|}^{2t/5}$ (see \cite[Page
  18]{GMP}). Therefore, by applying Theorem \ref{LiebeckShalev}, we
  deduce that 
  \[b_{X_{1}}(H_{1})=b_{X_1}(K_1) \leq 45 (\log |N|)/(\log |V|) + c.\]
  A slightly modified version of Theorem \ref{thm:boundedbK_1} gives
  $b_{X}(H) \leq 45 (\log |H|)/(\log |V|) + c$ for another universal
  constant $c > 0$.

  From now on assume that $\fX(N) = 1$. This means that $L=\fX(H)$
  acts faithfully on $\Pi$. Let $M$ be a normal subgroup of $H$
  strictly containing $\ker(\fX)$ such that $\fX(M)$ is a minimal
  normal subgroup of $L$ and let $\Delta$ be an orbit of $M$ on $\Pi$.
  Furthermore, let $M_\Delta:=\fX_\Delta(M)\nor L_\Delta$.  Notice
  that $\Delta\subseteq \Pi$ is an $H$-block of size at least $2$ and
  $M_\Delta$ is a direct product of isomorphic simple groups.

  Assume first that $S_\Delta\neq 1$. Then $S_\Delta$ is a subdirect
  product of the isomorphic (non-abelian) simple classical groups from
  the set $\{S_i\,|\,V_i\in\Delta\}$.

  If $M_\Delta$ centralizes $S_\Delta$, then all $\fX_i:S_\Delta\ra
  P\Gamma L(V_i)$ for $i\in\Delta$ are projectively equivalent since
  $M$ is transitive on $\Delta$. This contradicts our
  multiplicity-free assumption.  So we assume that $M_\Delta$ does not
  centralize $S_\Delta$. Since both $M_\Delta$ and $S_\Delta$ are
  normal subgroups in $L_\Delta$, this implies that $M_\Delta \cap
  S_\Delta \not= 1$. In particular $M_\Delta$ and $M_\Delta \cap
  S_\Delta$ are isomorphic to some powers of the (non-abelian) simple
  classical group $S$.  Since $M_\Delta$ is transitive on $\Delta$, we
  have that $|\Delta|\geq 5$ and $S_\Delta$ cannot contain a
  nontrivial, proper $M_\Delta$-invariant normal subgroup.  But
  $M_\Delta\cap S_\Delta\neq 1$ is normal in both $M_\Delta$ and
  $S_\Delta$, so $S_\Delta\leq M_\Delta$. Since any subnormal subgroup
  of $M_\Delta$ is normal in $M_\Delta$, we get that $S_\Delta$ is
  simple, so $S_\Delta\simeq S$ has linking factor $|\Delta|\geq 5$,
  which is in contradiction with the discussion following the proof of
  Proposition \ref{prop:linkingfact}.
  
  It remains to handle the case where $S_\Delta=1$. Then $L_\Delta$
  and $M_\Delta$ act faithfully and transitively on $\Delta$ and
  $M_\Delta$ is a normal subgroup of $L_\Delta$ isomorphic to a direct
  product of isomorphic simple groups. By Theorem \ref{thm1},
  $d_\Delta(L_\Delta) \leq 12$ or $\alt(\Delta)\leq L_\Delta \leq \sym
  (\Delta)$. 
	
	If $d_{\Delta}(L_{\Delta}) \leq 12$, then
  $b_{P(\Delta)}(L_\Delta) \leq 4$, by Lemma \ref{rem:dist}, and so
  $b_{V_\Delta}(L_\Delta) \leq 4$ (any subset of $\Delta$ can be
  represented by a vector in $V_\Delta$ whose projection to $V_i\in
  \Delta$ is non-zero if and only if $V_i$ is an element of the
  subset). Thus, $b_{X_\Delta}(N_H(\Delta))\leq
  b_{V_\Delta}(L_\Delta)+ b_{V_\Delta}(T_{V_\Delta})\leq 5$.  
  Applying Theorem \ref{thm:boundedbK_1} for $V_\Delta$ instead of $V_1$ 
  we get 
  \[b_X(H)\leq b_{X_\Delta}(N_H(\Delta))+1+\log 48+\frac{\log |P|}{\log|V|}\leq 
  \frac{\log|H|}{\log|V|}+12.\]
  
	Finally, if
  $d_{\Delta}(L_{\Delta}) > 12$, then $m := |\Delta| \geq 13$ and
  $\alt(\Delta)\leq L_\Delta \leq \sym (\Delta)$. In this case for any
  $V_i\in\Delta$, we have that $\fX_\Delta(H_{i}) \cong \alt ([m-1])$
  or $\fX_\Delta(H_{i}) \cong \sym ([m-1])$ must hold. But $S_i$ is a
  composition factor of $\fX_\Delta(H_{i})$ and it is a simple
  classical group. A contradiction.
\end{proof}
\subsection{Eliminating small tensor product factors from the $K_i$}
\label{sec:elim}
 Let us continue to use the notation of this section. 

The purpose of this subsection is to reduce the affine case of Theorem
\ref{Pyber} to the case when each $K_i$ acts on $V_i$ either as a
``big'' classical group (possibly over a field extension $\FF q$ of
$\FF p$) or as an alternating or symmetric group on the non-trivial
irreducible component of its natural permutation module. More
precisely, we will reduce the affine case of Theorem \ref{Pyber} to
the case where the action of $H$ is alternating-induced or
multiplicity-free classical-induced.  Since these types were dealt
with in the previous two subsections, this reduction will complete
the proof of Theorem \ref{Pyber} in the affine case.

\begin{lem}\label{lem:repeat}
  Let $L$ be a finite group and $W$ be a faithful, finite-dimensional
  $L$-module. For a positive integer $l$ let $V$ be the direct sum of
  $l$ copies of the $L$-module $W$. Then $b_{V}(L) = \lceil b_{W}(L)/l
  \rceil$.
\end{lem}
\begin{proof}
  Let $b':=b_{W}(L)$ and $\{ x_1,x_2,\ldots,x_{b'} \} \subset W$ be
  a minimal base for $L$ with respect to its action on $\Lambda$. Set $b:=\lceil
  b'/l \rceil$. Let us define the vectors
  \[y_1=(x_1,x_2,\ldots,x_l),\ y_2=(x_{l+1},x_{l+2},\ldots,x_{2l}),\ldots,
  y_{b}=(x_{(b-1)l+1},\ldots,x_{b'},0,\ldots,0)\in V.\] 
  It is easy to see that 
  $\{ y_1,\ldots,y_b \} \subset V$ is a minimal base for $L$ on 
  $V$.
\end{proof}
We now consider the case where the projective representation
$X_1:H_1\ra \Gamma L(V_1)$ preserves a proper tensor product
decomposition $V_1=U_1\otimes W_1$ over $\FF q$ where $U_1$ and $W_1$
are $\FF q $ vector spaces and $2\leq l:=\dim_{\FF q}(U_{1}) \leq
\dim_{\FF q} (W_{1})$.  Using that $H$ transitively permutes the
subspaces $V_1,\ldots,V_t$, it follows that each $X_i:H_i\ra \Gamma
L(V_i)$ preserves a corresponding tensor product decomposition
$V_i=U_i\otimes W_i$.

By taking the composition
of $X_i$ with the projection map to $W_i$, one can define new projective
representations $Y_i:H_i\rightarrow \Gamma L(W_i)$.  Let $Y:H\ra
GL(W(p))$ be the induced representation $Y = \ind_{H_1}^H(Y_1)$,
where $W$ can be identified with $W_1\oplus \ldots\oplus W_t$. The key
to our reduction argument is the following lemma, which gives an upper
bound for $b_X(H)$ in terms of $b_Y(H)$.
\begin{lem}\label{lem:tensor_elim}
  With the above notation we have $b_X(H)\leq \lceil b_Y(H)/l\rceil+4$.
\end{lem}
\begin{proof}
  By using a construction of Liebeck and Shalev (see the proof of
  \cite[Lemma 3.3]{LS1}), for each $1\leq i\leq t$ there exist 
  three vectors $v_1^{(i)},v_2^{(i)},v_3^{(i)}\in V_i$ such that 
  \[C_{GL(U_i)\otimes GL(W_i)}( v_1^{(i)},v_2^{(i)},v_3^{(i)} )
  \leq \id_{U_i}\otimes GL(W_i).\]
  Additionally, let $v_4^{(i)}=\alpha v_1^{(i)}$ for each $1\leq i\leq
  t$, where $\alpha$ is some generator of $\FF q^\times$.  Define
  $v_j=\sum_{i=1}^t v_j^{(i)}$ for $j=1,2,3$ and let $L:=C_H(
  v_1,v_2,v_3 )$. The choice of $v_1^{(i)}$ and $v_4^{(i)}$ guarantees
  that $X_i(L)\subset GL(U_i)\otimes GL(W_i)$ for each $i$, so 
  $X_i(L)\subset \id_{U_i}\otimes GL(W_i)$ by the displayed formula above.
  It follows that the restriction map 
  $X_i:L\cap H_i\ra \Gamma L(V_i)$ is projectively equivalent 
  to an $l=\dim_{\FF q}U_i$ multiple of $Y_i:L\cap H_i\ra \Gamma L(W_i)$. 
  
  Let $\Delta_1,\ldots,\Delta_s\subset \Pi$ be the orbits of $L$ on
  $\Pi$, $V_{\Delta_j} =\oplus_{V_i \in \Delta_j} V_i$ and $W_{\Delta_j}
  =\oplus_{V_i \in \Delta_j} W_i$ for every $1\leq j\leq s$.  Then each
  $V_{\Delta_j}$ is $X(L)$-invariant which means that
  $X=\oplus_{j=1}^s X_{\Delta_j}$ on $L$, where the
  $\pmod{T_{V_{\Delta_j}}}$-representation $X_{\Delta_j}: L\ra
  GL(V_{\Delta_j}(p))$ is defined by taking the restriction of $X(L)$
  to $V_{\Delta_j}$. One can similarly define the
  $\pmod{T_{W_{\Delta_j}}}$-representations $Y_{\Delta_j}:L\ra
  GL(W_{\Delta_j}(p))$ and establish the decomposition $Y=\oplus_{j=1}^s
  Y_{\Delta_j}$ on $L$. This means that if $V_{a} \in \Delta_j$ is
  arbitrary, then $X_{\Delta_j}=\ind_{L\cap H_a}^L(X_a)$ and
  $Y_{\Delta_j}=\ind_{L\cap H_a}^L(Y_a)$. Since $X_a$ on $L$ is
  projectively equivalent to the $l$ multiple of $Y_a$ on $L$,
  and induction of representations preserves multiplicity, we get that
  $X_{\Delta_j}$ is $\pmod{T_{V_{\Delta_j}}}$-equivalent to the
  $l$ multiple of $Y_{\Delta_j}$ on $L$ for every $1\leq j\leq
  s$. So, $X=\oplus_{j=1}^s X_{\Delta_j}$ is
  $\pmod{T_{V}}$-equivalent to the $l$ multiple of $Y$ on
  $L$. By using Lemma \ref{lem:repeat}, we get that $b_X(L)=\lceil
  b_Y(L)/l\rceil$. Since $b_X(H)\leq b_X(L)+4$ and $b_Y(L)\leq b_Y(H)$
  hold trivially, the result follows.
\end{proof}
\begin{cor}\label{cor:tensor_basesize}
With the above notation, if 
$b_Y(H)\leq c_1\cdot\frac{\log|H|}{\log|W|}+c_2$ for some constants
$c_1$ and $c_2\geq 10$, then $b_X(H)\leq c_1\cdot\frac{\log|H|}{\log|V|}+c_2$.
\end{cor}
\begin{proof}
By Lemma \ref{lem:tensor_elim} and by assumption,
\[
b_X(H)\leq \Big\lceil \frac{b_Y(H)}{l}\Big\rceil+4\leq
c_1\frac{\log |H|}{l\log |W|}+\frac{c_2}{l}+5\leq 
c_1\frac{\log |H|}{\log |V|}+c_2.
\]
\end{proof}
From now on we will assume that $K_1\leq GL(V_1)\simeq GL(k,p)$ is a
primitive irreducible linear group with unbounded base size. We may
make this assumption by Theorem \ref{thm:boundedbK_1}.

Primitive  groups  of  unbounded   base  size  were  characterized  in
\cite[Theorem 2]{LS1} and in \cite[Theorem 1, Proposition 2]{LS2}. In
the following we collect some of their properties in a form which will
be  most convenient  for us.  Note  that in \cite{LS1,LS2} the authors state
their theorem in terms of a tensor product of
several linear  groups, but for our  purpose it is  better to ``pack''
together all but the one with the largest dimension.

First we fix some further notation, mostly borrowed from
\cite{Guidici}. Let $U=U_k(p)$ be a vector space of dimension $k$
over $\FF p$. Let $H\leq GL(U_k(p))$ be a primitive
linear group. Let $q=p^f$ be the largest power of $p$ such that one can
extend scalar multiplication on $U$ to be an $\FF q$-vector space
$U=U_{k/f}(q)$ such that $H\leq \Gamma L(U_{k/f}(q))\leq GL(U_k(p))$.

If $\FF {q_0}$ is a subfield of $\FF q$, then $\mathrm{Cl}(r,q_{0})
\leq GL(r,q)$ denotes a classical linear group over $\FF {q_0}$ for
some subfield $\FF {q_0}\leq \FF q$ and for some $r\geq 9$. (This
lower bound on $r$ is assumed because we want to apply the result of
Section \ref{sec:classical}.)
\begin{thm}[Liebeck, Shalev \cite{LS1}, \cite{LS2}]\label{thm:tensorK_1}
  Let $H\leq GL(U_k(p))$ be a primitive linear group
  of unbounded base size and $q=p^f$ be maximal such that 
  $H \leq \Gamma L(U_{k/f}(q))$. Then there is a tensor product decomposition 
    $U=U_1\otimes U_2$ over $\FF q$ such that 
    $1\leq \dim (U_{1}) < \dim (U_{2})$ and $H$ preserves this tensor product 
    decomposition, that is, $H\leq N_{\Gamma L(U_{k/f}(q))}(GL(U_1)\otimes GL(U_2))$.
  Let $H^0=GL(U_{k/f}(q))\cap H$ and let $H^0_2$ be 
    the image of the projection of $H^0$ to $GL(U_2)$, that is,
    $H^0_2:=\{b\in GL(U_2)\,|\exists a\in GL(U_1):\ a\otimes b\in H^0\}$. 
    Then one of the following holds.
    \begin{enumerate}
    \item $H^0_2\simeq \sym(m)\times \FF q^\times$ or $\alt(m)\times \FF
      q^\times$ for some $m$ such that $U_2$ is the unique non-trivial
      irreducible component of the natural $m$-dimensional permutation
      representation of $\sym(m)$. In that case $\dim_{\FF
        q}(U_{2})=m-1$ unless $p\mid m$, when $\dim_{\FF q}(U_{2})
      =m-2$. 
    \item $H^0_2$ is a classical group $\mathrm{Cl}(r,q_{0})\leq
      GL(r,q)$ over some subfield $\FF {q_0}\leq \FF q$, where
      $r=\dim_{\FF q}(U_{2})$.
    \end{enumerate}
 \end{thm}
\begin{proof}
  This follows by combining parts of \cite[Theorem 1]{LS2} and
  \cite[Proposition 2]{LS2}.
\end{proof}

Note that there is a similar characterization of primitive linear
groups of large orders due to Jaikin-Zapirain and Pyber
\cite[Proposition 5.7]{JaPy}.

In the following we will apply Theorem \ref{thm:tensorK_1} to $K_i\leq
GL(V_{i})$ where $1\leq i\leq t$.  We can extend scalar multiplication
on each $V_i$ to become an $\FF q$-vector space for some $q=p^f$ to
get a tensor product decomposition $V_i=V_{i,1}\otimes V_{i,2}$
satisfying the statements of Theorem \ref{thm:tensorK_1}. In this way,
$V=V_s(q)$ becomes a vector space over $\FF q$ (where $sf=\dim_{\FF p}
(V)$) and $X:H\ra GL(V(p))$ is a $\pmod{T_V}$-representation of $H$ with 
$T_V\simeq \FF{q}^\times$.

We are now in a position to complete the proof of Theorem \ref{Pyber}
for affine groups.  In fact, we prove the following more general
statement for $\pmod{T_V}$-representations. To recover the original statement, take an irreducible imprimitive linear group $H \leq GL(V)$ with the identity. 

\begin{thm}
  There exists an absolute constant $c\geq 10$ such that if $X:
  H\rightarrow GL(V(p))$ is a $\pmod{T_V}$-representation of $H$ (with
  respect to some direct sum decomposition $V = \oplus_{i=1}^t V_i$) 
  induced from a primitive projective representation
  $X_1:H_1\ra \Gamma L(V_1)$, then
  \[
  b_X(H)\leq 45\frac{\log|H|}{\log|V|}+c.
  \]
\end{thm}
\begin{proof}
  By Theorem
  \ref{LieSha}, we may assume that $V$ is an imprimitive
  $X(H)T_V$-module, i.e. $t>1$.

  We proceed by induction on the dimension of $V_1$. Note that if
  $\dim V_1$ is bounded (or, more generally, if $b_{X_1}(H_1)$ is
  bounded), the theorem follows from Theorem \ref{thm:boundedbK_1}.

  By our assumption, $X_1(H_1)Z(GL(V_1))\leq \Gamma L(V_1)$ is a
  primitive semilinear group, so Theorem \ref{thm:tensorK_1} can be
  applied. Thus, an $\FF q$ vector space structure can be defined on
  each $V_i$ (where $\FF q$ is a (maybe non-proper) field extension of
  the base field of $V_i$) such that there is a tensor product
  decomposition $V_i=U_i\otimes W_i$ over $\FF q$ preserved by
  $X_i(H_i)$. Furthermore, $l:=\dim_{\FF q}(U_i)< \dim_{\FF q}(W_i)$.

  First, let us assume that the tensor product decomposition
  $V_i=U_i\otimes W_i$ is proper, i.e. $l\geq 2$.  Let $Y_i: H_i\ra
  \Gamma L(W_i)$ be the projective representation and $Y:H\ra
  GL(W(p))$ be the $\pmod{T_W}$-representation for $W=\oplus_{i=1}^t
  W_i$ defined in the paragraph before Lemma \ref{lem:tensor_elim}, so
  $Y=\ind_{H_1}^H(Y_1)$. By induction, $b_Y(H)\leq 45 \frac{\log
    |H|}{\log |W|} + c$ for some constant $c\geq 10$, so the result
  follows by Corollary \ref{cor:tensor_basesize}.  

  So we can assume that $l=1$. We can also assume that $\dim_{\FF
    q}V_i\geq 9$ by the second paragraph of this proof.  
  
  If $X_1(H_1)Z(GL(V_1))$ satisfies part (1) of Theorem \ref{thm:tensorK_1},
  then there is a (trivial) tensor product decomposition
  $V_1=U_1\otimes W_1$ with $\dim_{\FF q}U_1 = 1$ fixed by $X_1(H_1)$ and
  maps $\lambda_1:H_1\ra GL(U_1)\simeq \FF q^{\times}$ and $X_1':H_1\ra
  GL(W_1)$ such that $X_1'$ is a linear representation of $H_1$ and
  $X_1'(H_1)\simeq \sym(m)$ or $\alt(m)$. This means
  $X'=\ind_{H_1}^H(X_1'):H\ra GL(W)$ is an alternating-induced
  representation (where $W=\oplus_{i=1}^t W_i$), so
  $b_{X'}(H)\leq 2(\log{|H|}/ \log{|W|})+17$ by Theorem
  \ref{alternating}. Finally, $b_X(H)\leq b_{X'}(H)+4$ by Lemma
  \ref{lem:tensor_elim} and $|W|=|V|$, so $b_X(H)\leq 2(\log{|H|}/
  \log{|V|})+21$ and we are done. 

  For the remainder, we may assume that $X_1(H_1)Z(GL(V_1))$ satisfies part (2) of Theorem \ref{thm:tensorK_1},
  in which case $X$ is classical-induced.  In order to use Theorem
  \ref{classical} in this case, we need to further reduce it to
  satisfy the multiplicity-free condition. (For a reminder of this
  condition and the notation used in the rest of the proof, see Definition \ref{def:MultFree} and the preceding discussion at the start of Section 3.3.) For this purpose let $\Delta \subseteq \Pi$ be a maximal
  $H$-block violating the multiplicity-free condition,
  i.e. $|\Delta|\geq2$, $S_\Delta\simeq S$ and the representations
  $X_i:\widetilde{S_\Delta}\ra \Gamma L(V_i)$ for $V_i \in \Delta$ are
  all projectively equivalent. To simplify the notation, we may assume that
  $\Delta=\{V_1,V_2,\ldots,V_s\}$ with $s=|\Delta|>1$ and $k = \dim V_{1}$.  Let $X_\Delta:
  N_H(\Delta)\ra GL(V_\Delta(p))$ be the
  $\pmod{T_{V_{\Delta}}}$-representation defined by the restriction of
  $X$ (where $T_{V_{\Delta}}$ is defined by the decomposition
  $V_\Delta=\oplus_{V_i\in\Delta} V_i$).  Then
  $X=\ind_{N_H(\Delta)}^H(X_\Delta)$.

  Let $U_\Delta$ be an $s$-dimensional vector space over $\FF
  q$ with fixed basis $f_1,\ldots,f_s$ and let $W_\Delta$ be a $k$-dimensional vector space over $\FF q$ with fixed basis
  $e_1,\ldots,e_k$. Furthermore, let $\{b_1,\ldots,b_k\}$ be a basis
  of $V_1$.  By assumption, for each $2\leq i\leq s$ there are
  isomorphisms $\varphi_i:V_1\ra V_i$ and scalar maps
  $\lambda_i:\widetilde{S_\Delta}\ra \FF q^\times$ such that
  $X_i(h)=\lambda_i(h) \varphi_i X_1(h)\varphi_i^{-1}$ for every $h\in
  \widetilde{S_\Delta}$.  We also define $\varphi_1:=\id_{V_1}$ and
  $\lambda_1:\widetilde{S_\Delta}\ra \{1\}$.  Now,
  $\{\varphi_i(b_j)\,|\,1\leq i \leq s,\,1\leq j\leq k\}$ is a basis
  of $V_\Delta$.  Let $\Phi:V_\Delta\ra U_\Delta\otimes W_\Delta$ be
  the isomorphism defined by $\Phi(\varphi_i(b_j)):=f_i\otimes e_j$.
  By identifying $V_\Delta$ and $U_\Delta \otimes W_\Delta$ via
  $\Phi$, we get that for any $h\in \widetilde{S_\Delta}$, the matrix
  form of $X_\Delta(h)$ with respect to the basis $\{f_1\otimes
  e_1,f_1\otimes e_2,\ldots, f_s\otimes e_k\}$ is the Kronecker
  product of matrices $D(h)\otimes A(h)$ where $A(h)$ is the matrix
  form of $X_1(h)$ with respect to the basis $\{b_1,\ldots,b_k\}$
  while $D(h)$ is the diagonal matrix with entries
  $\lambda_1(h),\ldots,\lambda_s(h)$ in its main diagonal.  Since
  $X_\Delta(\widetilde{S_\Delta})$ is normalised by
  $X_\Delta(N_H(\Delta))$, we can apply \cite[Lemma 4.4.3(ii)]{KL} to see
  that $X_\Delta(N_H(\Delta))$ is contained in the Kronecker product
  of a group of monomial matrices and a group of matrices isomorphic
  to some classical group.

  This means that we have a tensor
  product decomposition $V_\Delta=U_\Delta \otimes W_\Delta$ preserved
  by $X_\Delta(N_H(\Delta))$. Taking the composition of $X_\Delta$
  with the projections to the factors of this tensor product
  decomposition, we can define the maps $Y_\Delta:N_H(\Delta)\ra
  GL(U_{\Delta})$ and $Z_\Delta:N_H(\Delta)\ra \Gamma L(W_{\Delta})$
  such that $Y_\Delta(N_H(\Delta))$ consists of monomial matrices,
  while $Z_\Delta(N_H(\Delta))$ is some classical group (modulo the
  group of scalar transformations). Then we can induce these
  representations to $H$ to get the monomial representation (with
  transitive permutation part) $Y=\ind_{N_H(\Delta)}^H(Y_\Delta)$ and
  classical-induced representation $Z=\ind_{N_H(\Delta)}^H(Z_\Delta)$.
  Note that $Z$ satisfies the multiplicity-free condition by the
  maximality of $\Delta$.  Furthermore, let $U:=\oplus_i
  U_{\Delta_i},\ W:=\oplus_i W_{\Delta_i}$, where $\{\Delta
  =\Delta_{1},\ldots,\Delta_{t/|\Delta|}\}$ is the orbit of $\Delta$
  under the action of $H$ on the power set of $\Pi$. Thus, $Y:H\ra
  GL(U(p))$ and $Z:H\ra GL(W(p))$.
  
  If $\dim U_{\Delta_{1}} \geq\dim W_{\Delta_1}$, then $b_Y(H)\leq
  \frac{\log|H|}{\log|U|}+10$ by use of Theorem \ref{thm:boundedbK_1}
  with $b=1$, so we get $b_X(H)\leq \frac{\log|H|}{\log|V|}+10$ by
  Corollary \ref{cor:tensor_basesize}.

  Similarly, if $\dim U_{\Delta_1} \leq \dim W_{\Delta_1}$ then $Z:H\ra GL(W(p))$ is
  multiplicity-free classical induced representation, so Theorem 
  \ref{classical} can be applied to conclude that
  $b_Z(H)\leq 45\frac{\log|H|}{\log|W|}+c$ for a suitable constant $c\geq 10$. 
  Using Corollary \ref{cor:tensor_basesize} again, we get that 
  $b_X(H)\leq 45\frac{\log|H|}{\log|V|}+c$ holds, which 
  completes our argument. 
\end{proof}

\section{Non-affine primitive permutation groups}

Pyber's conjecture is known to be true for all non-affine primitive
permutation groups. Since the explicit constants have not always been
specified, we collect here the information needed to complete the
proof of Pyber's conjecture with multiplicative constant $45$.

Let $G$ be a non-affine primitive permutation group acting on a finite
set $\Omega$ of size $n$. The first result deals with almost simple
groups.

\begin{thm}[Liebeck, Shalev \cite{LS99}; Burness et al \cite{Burness3}, \cite{Burness4}, \cite{Burness5}, \cite{Burness6}; Benbenishty \cite{Benbenishty}]
\label{almostsimple}
If $G$ is an almost simple primitive permutation group of degree $n$,
then $b(G) < 15 (\log |G| / \log n)$.
\end{thm} 

When $G$ is a primitive group of diagonal type, an almost exact formula for
$b(G)$ is determined by Fawcett \cite{Fawcett} (her upper bound differs from 
$b(G)$ by at most $1$ in every such case). Here we only need an
upper bound.

\begin{thm}[Gluck, Seress, Shalev \cite{GSSh}; Fawcett \cite{Fawcett}]
\label{diagonal}
If $G$ is a primitive permutation group of diagonal type and of degree
$n$, then $b(G) < (\log |G|/\log n) + 3 < 4 (\log |G|/\log n)$.
\end{thm}  

It remains to establish Theorem \ref{Pyber} when $G$ is a primitive
permutation group of product type or of twisted wreath product
type. For these types Pyber's conjecture has been proved by Burness
and Seress \cite{BS}. 

By the proof of \cite[Theorem 4.1]{BS}
it is sufficient to prove that if $G$ is of product type, then $b(G) <
(45/2) (\log |G| / \log n)$. 

Let $\Gamma$ be a finite set and let $H \leq \sym(\Gamma)$ be a primitive permutation group of almost simple type or of diagonal type. Let $\Omega$ be the direct product of $k$ copies of $\Gamma$ for some integer $k \geq 2$. For a transitive permutation group $P$ of degree $k$, the group $H \wr P$ acts in product action on $\Omega$ in a natural way. Let $G \leq \sym(\Omega)$ be a primitive permutation group contained in $H \wr P$. Assume that $\soc(G) = T^{k}$ where $T = \soc(H)$ and that the action of $G$ on the set of the $k$ direct factors of $\soc(G)$ is $P$. It remains to establish Pyber's conjecture with multiplicative constant $45$ for such groups $G$.  

If $H$ is of almost simple type, then \cite[Proposition 3.9]{BS} and its proof yield $b(G) <
(45/2) (\log |G| / \log n)$. Thus we may assume that $H$ is of diagonal type. In this case, by an argument different from the one in \cite[Proposition 3.10]{BS}, it is possible to obtain a bound with multiplicative constant less than $22$. For the details, see \cite[Section 4.3.2]{HLM} by Liebeck and the second and third authors.


\section*{Acknowledgement}
We thank Laci Pyber for drawing our attention to some of the
references in this paper. We also thank the referee for the very thorough reading of our paper and the large number of comments.


\end{document}